\documentclass[11pt,a4paper,reqno]{amsart}
\usepackage{amsmath,amssymb,graphics,epsfig,color,enumerate}
\usepackage{dsfont}  
\usepackage{verbatim}
\usepackage{epstopdf}

 \definecolor{refkey}{gray}{0.8}
 \definecolor{labelkey}{gray}{0.8}
\newtheorem{Theorem}{Theorem}[section]

\newtheorem{Lemma}[Theorem]{Lemma}
\newtheorem{Proposition}[Theorem]{Proposition}
\newtheorem{Corollary}[Theorem]{Corollary}
\newtheorem{remark}[Theorem]{Remark}

\newtheorem{definition}[Theorem]{Definition}
\newtheorem{conjecture}[Theorem]{Conjecture}

\newtheorem{assumptionA}{Assumption}

\definecolor{light}{gray}{.9}

\newcommand{\cD}{\ensuremath{\mathcal D}}

\newcommand{\cL}{\ensuremath{\mathcal L}}

\newcommand{\bbE}{{\ensuremath{\mathbb E}} }

\newcommand{\bbP}{{\ensuremath{\mathbb P}} }

\newcommand{\bbR}{{\ensuremath{\mathbb R}} }

\newcommand{\bbZ}{{\ensuremath{\mathbb Z}} }

\let\a=\alpha    \let\d=\delta  \let\e=\varepsilon
 \let\g=\gamma     \let\k=\kappa  \let\l=\lambda
            
   \let\t=\tau

\let\O=\Omega      

\begin{document}
 
\title[A class of random walks in reversible environments]{A class of random walks in reversible   {dynamic} environments: antisymmetry and applications to the East model}
%
%

\author{L. Avena}
\address{Luca Avena. Mathematisch instituut
Universiteit Leiden. Postbus 9512
2300 RA Leiden,
The Netherlands}
\email{l.avena@math.leidenuniv.nl}
\author{O. Blondel}
\address{Oriane Blondel. CNRS, Univ. Lyon, CNRS UMR 5208, Institut Camille Jordan, 43 blvd. du 11 novembre 1918, F-69622 Villeurbanne cedex, France}
\email{blondel@math.univ-lyon1.fr}
\author{A. Faggionato}
\address{Alessandra Faggionato. Dipartimento di Matematica, Universit\`a di Roma La Sapienza.
  P.le Aldo Moro 2, 00185 Roma, Italy}
\email{faggiona@mat.uniroma1.it}

\begin{abstract} 
We introduce via perturbation a class of random walks in reversible dynamic environments having a spectral gap. In this setting one can apply the mathematical results derived in \cite{mainprobe}. As   { first results}, we show that the asymptotic velocity is antisymmetric in the perturbative parameter   { and, for a subclass of random walks, we characterize  the velocity and a stationary distribution of the  environment seen from the walker   as suitable series in the perturbative parameter}. We then consider as a special case a random walk on the East model that tends to follow dynamical interfaces between empty and occupied regions. We study the asymptotic velocity and 
density profile for the environment seen from the walker. In particular, we determine the sign of the velocity when the density of the underlying East process is not $1/2$, and we discuss the appearance of a drift in the balanced setting given by density $1/2$. 
\end{abstract}

\maketitle


\section{Introduction}
In  \cite{mainprobe}  we studied continuous-time random walks in dynamic random environments on the $d-$dimensional integer lattice, $d\geq 1$, in a perturbative regime. 
More precisely, we considered a stationary  Feller Markov process, playing the role of the environment  and satisfying the Poincar\'e inequality\footnote{For a reversible process,  the Poincar\'e inequality   is equivalent to the positive   spectral gap of the generator.}. In addition, we considered 
 a random walk with  transition rates given by    functions of the (autonomously) evolving environment.
The  main assumption required  that the random walk is a small perturbation either of an  homogeneous or of a ``stationary" walk, the latter meaning that  the   environment viewed from the walker has the same stationary distribution as   the environment itself.  
 In this setting, we characterized the ergodic behavior of the   environment viewed  from  the walker, and we derived a law of large numbers (i.e. existence of an asymptotic non random velocity)  and an invariance principle (i.e. gaussian fluctuations under diffusive rescaling)  for the random walk. One main tool there 
was   the  derivation and a careful analysis of a   series expansion of Dyson--Phillips type  for the semigroup associated with  the environment as seen from the walker.  We  review in Section \ref{contratto} the main results of \cite{mainprobe} that will be  used in the rest  of the paper (cf. in particular Theorem \ref{chorale} below). 

 
We aim here to    illustrate how the results of \cite{mainprobe} can give non-trivial information about  random walks in dynamic random environments, beyond their  diffusive behavior. Environments here will be reversible stochastic    particle systems on $\bbZ^d$  with a positive spectral gap (in particular, the environment at a given time is a configuration in $\{0,1\}^{\bbZ^d}$).   We first introduce in Section \ref{maritozzo}  a  class of random walks with transition rates satisfying suitable algebraic identities  and show 
a hidden antisymmetry relation in the asymptotic velocity (cf. Theorem \ref{guepe}). More precisely, if $v(\e)$ is the asymptotic velocity  at  perturbative parameter $\e$, we have the antisymmetry relation
\begin{equation*}
v(\e)=-v(-\e).
\end{equation*}
As discussed in Sections~\ref{maritozzo} and ~\ref{sec:east}, this is not a consequence of obvious symmetries in the system. A special example of random walk in the above mentioned class is given by   what we call ``$\e$-RW'', a one-dimensional random walk with drift $2\e$ (resp. $-2\e$) on top of particles (resp. on empty sites). This type of random walk (with different drifts on top of empty/occupied sites) 
has been recently studied in \cite{AFJV,AdHR1,AdHR2,ASV,HdHSST,dHKS,dHS,HS,MV} for different choices of environments. The interest is due to the fact that (in case of opposing drifts) it represents one of the simplest example of a random walk with space-time inhomogeneous random transitions that can give rise to some  slow-down or ``trapping" effects (cf. \cite{AT})
 similar to the well known phenomenology in 1-dimensional static random environment \cite{BG,S,Z}. 
Under our assumptions, the trapping effect does not occur on the diffusive scale, but the $\e$-RW favors spending time oscillating between a particle and a hole, and so tends to lie at interfaces between occupied/empty regions. Its behavior is therefore connected to space-time correlations in the environment, which can be difficult to grasp. We first derive two main results for generic  $\e$--RW's: a deeper analysis of the series expansion for its asymptotic velocity $v(\e)$ (see Proposition  \ref{signature}) and for the limiting distribution of the environment viewed from the walker (see Proposition \ref{denso}). 
 
We then study in more detail the $\e$-RW on the East model. The latter has been introduced in the physics literature  as a simplified  model for glassy systems \cite{JE}, and belongs to the class of kinetically constrained model \cite{BB}.  It has  received much attention within the physics and mathematics communities, since it catches some relevant features of glassy dynamics as  e.g. aging, dynamical heterogeneity, huge relaxation times (cf. \cite{CMRT,CFM,FMRT, eastrecent,SE1} and references therein). Of particular interest to both the physics and mathematics communities is the structure of the space-time correlated ``bubbles'' of occupied sites (see Figure~\ref{EastFA}) and tracing the $\e$--RW on the East model allows to catch some information on these bubbles. We stress that the East model  has a positive spectral gap  \cite{AD,CMRT} but does not display any uniform mixing property or attractiveness. Therefore we can only use the results of \cite{mainprobe} and not for instance those of \cite{AdHR1,AdHR2, RV}.

 For the $\e$-RW on the East model, we discuss evidence of a negative asymptotic  drift in the balanced case of density $1/2$ and  give two  theoretical results supporting this fact in addition to simulations (cf. Propositions \ref{shaun} and \ref{DegDrift}).
It is tempting to interpret the sign of the asymptotic velocity as a signature of the orientation of the East model, but indeed we can show that the velocity remains the same if one replace the East model by the West model, which has the opposite orientation (cf.  Corollary \ref{EW}). Finally, in Corollary \ref{oscar} we give a detailed analysis of the density profile of the limiting distribution of the East model viewed from the $\e$--RW.

Let us notice that the study of the $\e$--RW on the East model  was partially inspired by  \cite{JKGC}, where the authors consider  random walks on the FA1f model (the symmetric version of the East model). 
An  investigation  based on the expansion derived in \cite{mainprobe} might be performed as well  for other types of random walks as e.g.  the ones considered in \cite{JKGC}. 
A further  study of random walks in kinetically constrained models  is given in  \cite{O3}.

  Finally, we mention that  the negative drift for the $\e$--RW on the East model and  our qualitative analysis of the   density profile of the East model  viewed from the $\e$--RW  are supported by numerical simulations performed by Philip Thomann.

\section{Preliminaries}\label{contratto}
\subsection{Definitions}

We consider a   {Feller}  Markov process $(\xi_t)_{t\geq 0}$ on $\O:=\{0,1\}^{\bbZ^d}$ with generator $\mathcal{L}_{\rm env}$ which can be thought of as an interacting particle system,   {playing the role of dynamic
  random environment}. For $(x,t)\in\bbZ^d\times[0,\infty)$, if $\xi_t(x)=1$ we say that there is a particle at time $t$ at position $x$, else site $x$ is considered empty at time $t$   {(equivalently, there is a hole at $x$ at time $t$)}. See \cite{L} for a standard reference on this type of Markov processes. We write $\bbE^\mathrm{env}_\xi$ for the expectation w.r.t. the law of the environment started from $\xi$ and $\tau_x$ for the translation operator on $\bbZ^d$ such that $\tau_x\xi(y)=\xi(x+y)$ for $x,y\in\bbZ^d, \xi\in\O$.

\begin{assumptionA}\label{piqure} We assume the following properties for the dynamic environment:
\begin{itemize}
\item[(i)](Reversibility) $(\xi_t)_{t\geq 0}$ admits a reversible translation invariant probability measure $\mu$ on $\O$.
\item[(ii)](Positive spectral gap) The generator $L_{\mathrm{env}}$ has $0$ as simple eigenvalue and the rest of its spectrum is in $[\gamma,+\infty)$ for some $\gamma>0$.
\item[(iii)] The Markov semigroup $S_{\mathrm{env}}(t)$, with $S_{\mathrm{env}}(t)f(\xi):=\bbE_\xi^{\mathrm{env}}[f(\xi_t)]$, commutes with spatial translations, {i.e.} $S_{\mathrm{env}}(t)(f\circ\tau_x)=(S_{\mathrm{env}}(t)f)\circ\tau_x$ $\mu$--a.s. for any local function $f$ and $x\in\bbZ^d$.
\end{itemize}  
\end{assumptionA}

We point out that Assumption~\ref{piqure}-(ii) is equivalent to the so-called Poincar\'e inequality:
$\g \|f\|^2 \leq - \mu ( f \mathcal{L}_\mathrm{env}f ) $  for all $ f \in \cD (\mathcal{L}_\mathrm{env})$ with $  \mu(f)=0$,
where $\cD (\mathcal{L}_\mathrm{env})$ denotes the domain of the operator $\mathcal{L}_\mathrm{env}$. 

We interpret the process $(\xi_t)_{t\geq 0}$ as a   {dynamic random environment} for a continuous-time random walk  $(X_t^{(\e)})_{t\geq 0}$ on $\bbZ^d$ which starts at the origin and that we now define. The rate for a jump from $x$ to   {$x+y$} when the environment is equal to $\xi$ will be denoted by $r_\e(y,\tau_x\xi)$. Here, $\e$ is a perturbative parameter, whose precise meaning we explicit in    {Subsection \ref{3} below}.

\begin{assumptionA}\label{lione}We assume that for suitable functions $r(y,\eta)\geq 0$, $\hat{r}_\e(y,\eta)$ with finite support in $\eta$ and finite range in $y$\footnote{More precisely, we assume that there exists $R$ such that for $|y|\geq R$, $r(y,\cdot)\equiv 0$ and $\hat{r}_\e(y,\cdot)\equiv 0$ and for all $y$, $r(y,\cdot),\hat{r}_\e(y,\cdot)$ have finite support.}, the jump rates admit the decomposition
\[r_\e(y,\eta)=r(y,\eta)+\hat{r}_\e(y,\eta).\]
Moreover we assume that 
\begin{equation}\label{revrate}
r(y,\eta)=r(-y,\tau_y\eta).
\end{equation}
\end{assumptionA}

Note that $\hat{r}_\e$ should be considered as a perturbative contribution to the transition rates $r_\e$, so that $r(y,\eta)$ can be thought of as the transition rates for an unperturbed random walk. Then, since $\mu$ is translation invariant, the last assumption \eqref{revrate} is the detailed balance condition and is equivalent to the reversibility of $\mu$ for the environment seen from the unperturbed walker. 
 
Due to dependence on the environment, such a random walk is not Markovian itself, but 
the joint process $(\xi_t, X_t^{(\e)})_{t\geq 0}$ on state space $\O\times \bbZ^d$ is a Markov process with generator  
\begin{equation}\begin{aligned}\label{RWRE}
\mathcal{L}_{\rm rwre}^{(\e)} f(\xi,x) &:= \mathcal{L}_{\rm env} f(\xi,x)+ \sum_{y \in \bbZ^d}r_\e(y,\tau_{x}\xi) \big[f(\xi,x+y)-f(\xi, x)\big],
\end{aligned}\end{equation} 
where the operator $\mathcal{L}_{\rm env}$ acts only on the first coordinate of $f$.

Later, we will consider more closely the following one-dimensional special case, which we call $\e$-RW.

\begin{center}
\begin{figure}
\includegraphics[scale=.3]{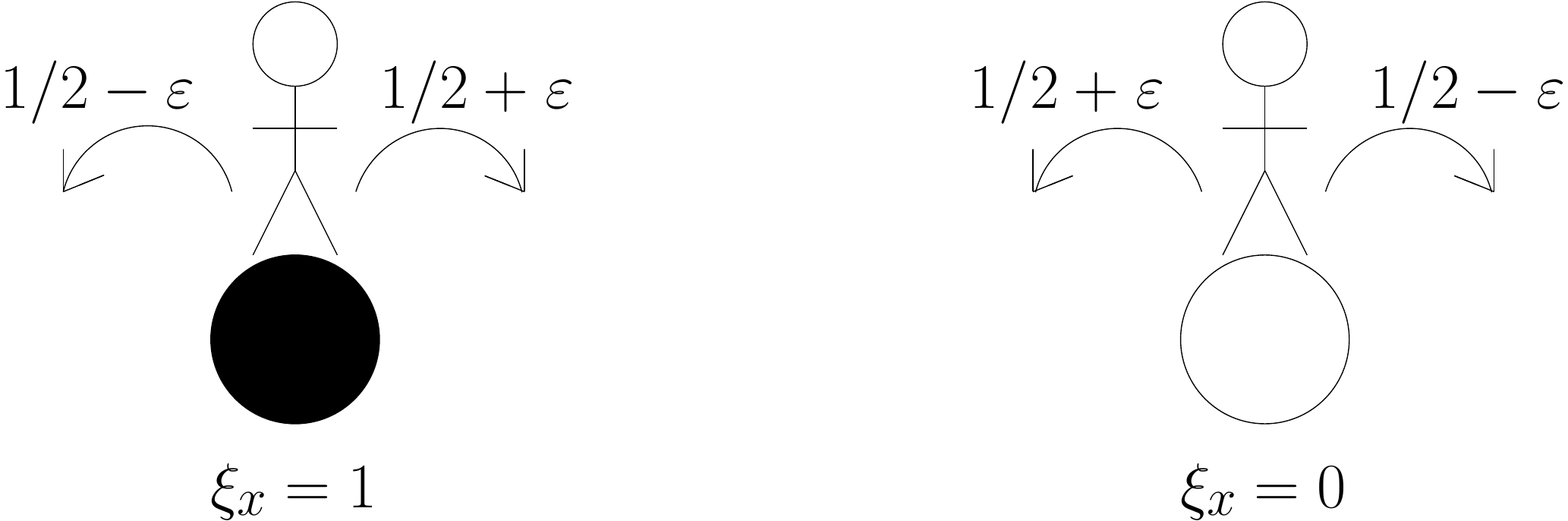}
\caption{A summary of the jump rates for the $\e$--RW.}
\label{ansatz}
\end{figure}
\end{center}

\begin{definition}\label{RWp}{\bf ($\e$-RW)}
For $\xi\in\{0,1\}^\bbZ$ and fixed $\e\in[-1/2,1/2]$, set (see Figure~\ref{ansatz})
\begin{equation}\label{pRate}
  {r_\e(y,\xi)}:=
\left\{
\begin{array}{cl}
1/2+\e(2\xi(0)-1) & \text{if} \quad y=+1\,,\\
1/2-\e(2\xi(0)-1) & \text{if} \quad y=-1\,,\\
0&\text{otherwise.}
\end{array}
\right.
\end{equation}
\end{definition}

Here, the perturbative role of $\e$ is clear, and in this case, the unperturbed random walk ($\e=0$) is the simple symmetric random walk.

\subsection{Environment seen by the walker}\label{3}
One of the most common approach to study random motion in random media is to analyze the so called {\em environment seen by the walker}, that is, the Markov process $(\eta^{(\e)}_t)_{t\geq 0}$ with state space $\O$ defined by $\eta_t^{(\e)}:=\tau_{X^{(\e)}_t}\xi_t$, with associated generator given by 
\begin{equation}\begin{aligned}\label{EnvFromRW}
\mathcal{L}_{\rm ew}^{(\e)} f(\eta) := \mathcal{L}_{\rm env}f(\eta)+\sum_{y \in \bbZ^d}r_\e(y,\eta)\big[f(\tau_{y}\eta)-f(\eta)\big]\,,\;\; \eta\in\O\,.
\end{aligned}
\end{equation}
Note that the jumps of the walker in \eqref{RWRE} turn into spatial--shifts for the environment seen by the walker. We write $(S_\e(t))_{t\geq 0}$ for the semigroup associated with this Markov process $(\eta_t^{(\e)})_{t\geq 0}$. When $\e=0$ we simply write $S(t)$.

In the following theorem we recall some results from \cite{mainprobe} that are relevant to our discussion. 
We set
\begin{equation}\label{Lpe}\hat L_\e f(\eta):=\mathcal{L}_{\rm ew}^{(\e)}-\mathcal{L}_{\rm ew}^{(0)}=\sum_{y \in \bbZ^d}\hat{r}_\e(y,\eta)\big[f(\tau_{y}\eta)-f(\eta)\big] .\end{equation}
Due to our assumptions, $\hat{L}_\e$ has bounded norm $\|\hat{L}_\e\|$ as operator in $L^2(\mu)$.
  {For example, for the $\e$--RW the operator $\hat L_\e$ is given by
\begin{equation} \label{michele}
\hat L_\e f (\eta)= \e ( 2 \eta(0)-1)\bigl[ f( \t_1 \eta)- f(\t_{-1} \eta)\bigr]\,.
\end{equation}}
  {
Note that in this case $\|\hat{L}_\e\|$ is bounded from above  by $2\epsilon$. }

  {
As the reader will see, our results  hold for $\epsilon$ such that $\|\hat{L}_\e\|< \gamma$ 
(cf. Assumption  \ref{piqure}--(ii)).  If, as in the examples discussed below, $ \|\hat{L}_\e\|=C \e $, 
this trivially means that we restrict to $\epsilon $ small. Since  interesting perturbations are not necessarely
explicitly linear in the perturbative parameter, we keep the more general condition $\|\hat{L}_\e\|< \gamma$.}

\begin{Theorem}\label{chorale}{\cite{mainprobe}}
Under Assumptions \ref{piqure} and \ref{lione} and further assuming that    { $\|\hat{L}_\e\|<\gamma$}, the following holds:
\begin{itemize}
\item[(i)] The process $(\eta^{(\e)}_t)_{t\geq 0}$ admits a unique probability measure $\mu_\e$ which is invariant and absolutely continuous w.r.t. $\mu$.   {Moreover, $\mu_\e$ is time ergodic}.  The distribution of $\eta^{(\e)}_t$ converges to $\mu_\e$ as $t\rightarrow\infty$ if the distribution of $\eta_0^{(\e)}$ is absolutely continuous w.r.t. $\mu$. 
\item[(ii)]If in addition $r(y,\eta)>0\Rightarrow r_\e(y,\eta)>0$,   {then} $\mu_\e$ and $\mu$ are mutually absolutely continuous.
\item[(iii)]The measure $\mu_\e$ admits the following representation: for every $f\in L^2(\mu)$
  \begin{equation}\label{muinfinity_bis}
\mu_\e(f) = \mu(f)+\sum _{n=0}^\infty \int_0^\infty \mu\left( \hat L_\e S_\e^{(n)}(s)f\right) ds\,,
\end{equation}
where the integrals and the series are absolutely convergent and the operators $S_\e^{(n)}(s)$, $n\geq 0$, are defined iteratively as 
\begin{equation}\label{Sn}
S_\e^{(0)}(t)f:= S(t)f, \quad S_\e^{(n+1)}(t)f:=\int_0^t  S(t-s) \hat L_\e S_\e^{(n)}(s)f ds \,.
\end{equation} 
Moreover,   {it holds 
\begin{equation}\label{ikea} \Big|\int_0^\infty \mu\left( \hat L_\e S_\e^{(n)}(s)f \right) ds\Big |
\leq   { (\| \hat L_\e \| }/\g)^{n+1} \| f-\mu(f) \|\,.
\end{equation} }
\item[(iv)] For $\eta\in\O,$ we introduce the local drift $j^{(\e)}(\eta):=\sum_y yr_\e(y,\eta)$ and set $v(\e):= \mu_\e (j^{(\e)})$. Then for $\mu_\e$--a.e. $\xi$
\begin{equation}\label{velo}
\frac{1}{t}X^{(\e)}_t\underset{t\rightarrow\infty}{\longrightarrow}v(\e)\quad a.s.
\end{equation}
In particular, $v(\e)$ can be written as
\begin{equation}\label{Expvelo}
v(\e)=\mu(j^{(\e)})+\sum_{n=0}^\infty\int_0^\infty \mu( \hat L_\e S_\e^{(n)}(s)j^{(\e)}  )ds\,.
\end{equation}
\end{itemize}
\end{Theorem}

To obtain the above theorem, not all our assumptions are necessary. We refer the interested reader to Theorem  2 in \cite{mainprobe},   {where a more general statement is given, and to (42) in \cite{mainprobe}  which allows to get \eqref{ikea}. Strictly speaking, in the assumptions of   Theorem \ref{chorale} one should include that the environment process has a non pathological generator $\cL_{\rm env}$ (see Prop. 3.1 in \cite{mainprobe} for a precise statement), anyway this additional technical assumption is satisfied in all standard models. Finally,  we point out that a perturbative characterization of the stationary distributions $\mu_\e$ is given in \cite{KO} in a different form}. 


\section{A class of RW with an antisymmetry property}\label{maritozzo}
\subsection{Antisymmetry relation for the velocity}

We can now state   { a first }  new result given by an antisymmetry relation for the velocity of the random walks in dynamic environment introduced above. To this aim, we introduce some additional assumptions.

\begin{assumptionA}\label{brioche}   {The following identities are satisfied:}
\item[(i)]$\sum_y yr(y,\cdot)\equiv 0$
\item[(ii)] $\hat{r}_\e=-\hat{r}_{-\e}$
\item[(iii)]$\hat{r}_\e$ can be factorized into $\hat{r}_\e(y,\eta)=\a(y)\bar{r}_\e(\eta)$ with $\a$ antisymmetric, that is 
$\a(y)=-\a(-y)$.
\end{assumptionA}

As an example one may think of $\bar{r}_\e(\eta)$ as $\e$ times the number of particles in a given neighborhood of the origin. In that case, for $\e>0$, each particle in the environment that falls in the ``vision field'' of the walker favors jumps $X_t^{(\e)}\rightarrow X_t^{(\e)}+ y$ when $\alpha(y)>0$ and discourages them when $\alpha(y)<0$ (and vice-versa for $\e<0$). 
  {Another example (one--dimensional) is given by 
 the $\e$--RW, which indeed  satisfies both Assumptions \ref{lione} and \ref{brioche}}.

\begin{Theorem}\label{guepe}
Assume Assumptions~\ref{piqure}, \ref{lione}, \ref{brioche} and $\|\hat{L}_\e\|<\gamma$, then
\begin{equation}\label{antisym}
v(-\e)=-v(\e).
\end{equation}
\end{Theorem}

We prove this theorem in Section~\ref{proofguepe}. For the sake of clarity we restrict ourselves there to the case of the $\e$--RW, but the proof extends easily to the general case.

We point out that $X^{(-\e)}$ is not a time-reversed version of $X^{(\e)}$. Indeed the trajectories of these two processes are quite different in general (see Figure~\ref{twowalkers} for an illustration   {in the case of the $\e$--RW on the East model, which will be explained in detail below}). In particular, Theorem \ref{guepe} does not follow by taking the time reversion of $X^{(\e)}$. This is further discussed in Section~\ref{feller}.

Let us explain  the difficulty behind Theorem \ref{guepe}. We first observe that Assumption \ref{brioche} trivially implies $j^{(\e)}(\eta)=-j^{(-\e)}(\eta)$. This identity alone is not enough to prove the antisymmetry relation \eqref{antisym} since, due to Theorem~\ref{chorale}, $v(\e)=\mu_\e(j^{(\e)})$, while   {$v(-\e)=\mu_{-\e}(j^{(-\e)})$}. One could therefore ask whether the antisymmetry relation could be due to a possible equality of $\mu_\e$ and $\mu_{-\e}$. However, as illustrated for instance in \eqref{limitdensity}   {below}, the two probability distributions $\mu_\e$, $\mu_{-\e}$ do not coincide in general.

\subsection{Velocity and density profile for the $\e$--RW}
  When applying Theorem~\ref{chorale}, we get the following more refined results in the case of the $\e$--RW.

\begin{Proposition}\label{signature}
Under Assumptions~\ref{piqure} and for $2|\e|<\gamma$, the asymptotic velocity $v(\e)$ of the $\e$--RW can be  expressed as
\begin{equation}\label{limitvelocity}
v(\e)=
\begin{cases}
2\e (2\mu(\eta(0))-1) + O(\e^3) & \text{ if } \mu(\eta(0)) \not = 1/2\,\\\e^3 \k + O(\e^5) & \text{ if } \mu(\eta(0)) = 1/2,\,
\end{cases}
\end{equation}
with
\begin{equation}\label{k}
\k:= -8  \mu\left ( (2\eta(0) -1) \Big\{\int_0^\infty\bbE^{(0)}_{\eta}  [\eta_s(1)-\eta_{s}(-1)]ds  \Big\}^2 
\right)\,,
\end{equation}
  {where  the expectation $\bbE^{(0)}_{\eta} $ refers to  the environment viewed from simple random walk,  when starting at $\eta$}. 
Moreover, the even terms in the expansion \eqref{Expvelo} equal zero and the antisymmetry relation $v(\e)=-v(-\e)$ holds.
\end{Proposition}
When $\mu(\eta(0)) \not = 1/2$, for $\e$ small enough, the sign of the velocity can be read from \eqref{limitvelocity}.
When $\mu(\eta(0))= 1/2$, the scenario is more subtle.   {Since $\nu_{1/2}$ is left invariant by particle--hole exchange and  due to the form  } of transitions of the   {$\e$--RW},
 one may naively guess that the velocity is zero.
Despite this guess, the answer seems strongly dependent on the specific dynamics of the underlying environment.
In   {Section \ref{sec:east} }   we investigate more precisely the case of the East model and we give arguments supporting the negativity of $v(\e)$ for $\e>0$   {at density $1/2$}. Let us conclude this section by observing that in simple settings it is easy to deduce that $v(\e)\equiv 0$ at density $1/2$.

\begin{remark}[{\bf   {Zero} velocity for independent spin-flip dynamics}]
Suppose that the dynamics of the environment is given by independent spin-flips   { with generator} 
$\mathcal{L}_{\rm env} f(\eta)=\gamma\sum_{x\in \bbZ} [f(\eta^x)-f(\eta)]$,   {where $\eta^x$ is the configuration obtained from $\eta$ by a spin flip at $x$}.   {Then}  the product Bernoulli measure with density $1/2$ is ergodic and reversible for this dynamics. Since for this initial distribution the process is invariant by inversion of particles and holes, it is easy to see that $v(\e)=v(-\e)$. Consequently, by the antisymmetry relation \eqref{antisym}, we conclude that $v(\e)\equiv 0$.
\end{remark}
Our next result provides some description of the density profile of the   {environment}  observed by the $\e$-RW.
\begin{Proposition}\label{denso}
Under Assumptions~\ref{piqure} and for $2|\e|<\gamma$, the stationary distribution  $\mu_\e$ of the   {environment   seen from the $\e$--RW  (i.e. of the process  $(\eta^{(\e)}_t)_{t\geq 0})$} admits the following representation for any function $f\in L^2(\mu )$:
\begin{equation}\label{boh}
\mu_\e(f)=\mu (f)+2\e\int_0^\infty ds\sum_{y\in \mathbb Z}p_t(y)\mu\big(\xi(0)\bbE^{\rm env}_\xi\left[f(\tau_{y+1}\xi_s)-f(\tau_{y-1}\xi_s)\right]\big)+O(\e^2),
\end{equation}
where $\bbE^{\rm env}_{\xi}$ denotes the expectation w.r.t.\@ the environment with generator $\mathcal{L}_{\rm env}$  starting from $\xi\in \O$, and $p_t(y-x)$ is the probability that a simple symmetric random walk jumping at rate $1$, started at $x$, is in position $y$ at time $t$.
\end{Proposition}
 
\section{The $\e$-RW on the East model}\label{sec:east}
\begin{definition}\label{EastDef}{\bf (East dynamics)}
For $x\in\bbZ$ and $\xi\in\Omega$, set $c^{\rm east}_x(\xi):= 1- \xi (x+1) $. 
The East model is the Markov process on $\{0,1\}^\bbZ$ with infinitesimal generator 
\begin{equation}\label{EastGen}
\mathcal{L}_{\rm east} f(\xi):=\sum_{x\in\mathbb{Z}}c_x ^{\rm east} (\xi)\left[\rho(1-\xi(x))+(1-\rho)\xi(x)\right]\left[f(\xi^x)-f(\xi)\right],
\end{equation}
with $\rho\in(0,1)$ being a fixed parameter, and 
$\xi^x$ the configuration obtained by flipping the coordinate of $\xi$ at site $x$. 
\end{definition}

The East model can be described as follows: at each site $x$, after an exponential time of parameter $1$ and provided that the kinetic constraint $ c_x ^{\rm east} =1$ is satisfied\footnote{The East model belongs to the class of kinetically constrained spin models \cite{CMRT}.}, the particle configuration $\xi(x)$ is refreshed and set equal to $1$ with probability $\rho$ and equal to $0$ with probability $1- \rho$.
It is simple to check that the Bernoulli product measure with density $\rho$, denoted by $\nu_{\rho}$, is a reversible probability  measure. 

\begin{remark}[{\bf  West and FA-1f models}]\label{WestFA}
We notice that by definition of $c_x ^{\rm east}$, in order to change the state at site $x$, the site to its ``East", i.e. at position $x+1$, has to be vacant.
This justifies the name of this model. The West model is the process with  generator as in \eqref{EastGen} when we replace $c^{\rm east}_x(\xi)$ by $c^{\rm west}_x(\xi):= 1- \xi (x-1)$, which means that the constraint has to be satisfied in the other direction.
A symmetrized version of the East and West models is the so-called FA1f (Fredrickson-Andersen one spin facilitated) model, that is the process with generator as in \eqref{EastGen} with $c^{\rm fa}_x(\xi):=1- \xi (x-1)\xi (x+1)$. In particular, for any $\rho\in(0,1)$, the Bernoulli product measure $\nu_{\rho}$ is again reversible for the West and the FA1f models. Figure \ref{EastFA} shows space-time realizations of the East and the FA1f particle systems. The West looks like the East reflected w.r.t. the time axis.
\end{remark}

All the models introduced above satisfy Assumption~\ref{piqure} for $\rho\in (0,1)$ (see \cite{AD,CMRT} for the positivity of the spectral gap), and we can therefore apply the results of \cite{mainprobe} to them. However, stronger inequalities of Sobolev type do not hold for the East model \cite{eastrecent} and we can hope for no uniform mixing property due to the hardness of the constraint in Definition~\ref{EastDef}. It is also non-attractive, as one can check easily by noticing that more empty sites allow to create more holes, but also to add more particles.

In Figure~\ref{EastFA} we present a simulation of the East/FA1f dynamics. One can observe bubbles of occupied sites forming. These are space-time regions with zero activity and a fundamental feature of kinetically constrained dynamics. Rigorous attempts towards an understanding of their structure can be found in \cite{BT,BLT, O2, GLM}.

\begin{figure}
\begin{center}
\includegraphics[scale=.5]{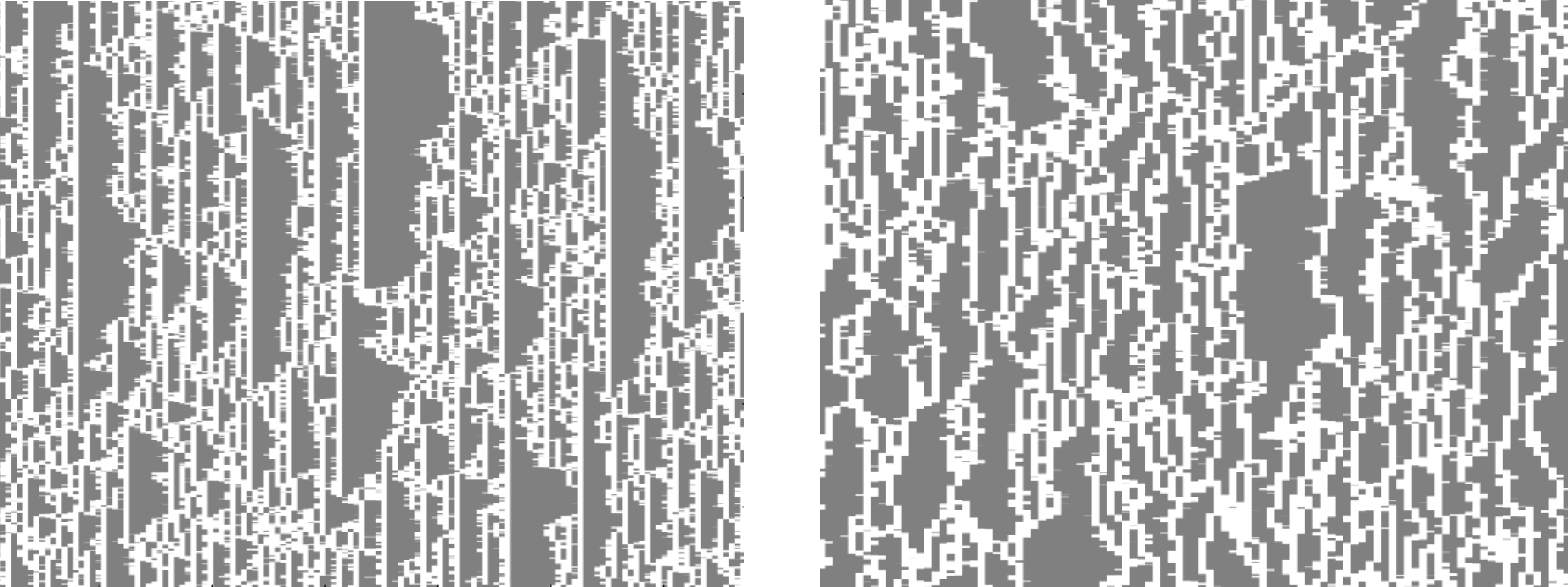}
\caption{Two space-time (horizontal-vertical axis) realizations of the East and the FA1f models, left and right pictures, respectively. Particles ($1's$) in gray.}
\label{EastFA}
\end{center}
\end{figure} 
\subsection{Asymptotic velocity}\label{feller}
 
Simulations suggest that the $\e$-RW is ballistic for $\rho=1/2$, drifting to the left when $\e>0$ (see Figure \ref{pSpeed}). 
This motivates the following conjecture: 
\begin{conjecture}\label{conj}
When the environment is the East model at density $1/2$, for $\e>0$ (resp. $\e<0$) we have $v(\e)<0$ (resp. $v(\e)>0$).
\end{conjecture}
Below we will give two theoretical arguments supporting the above conjecture, based on Propositions~\ref{shaun} and \ref{DegDrift}.
 
\begin{figure}[hbtp]
\vspace{0.5cm}
\begin{center}
\includegraphics[width=4cm]{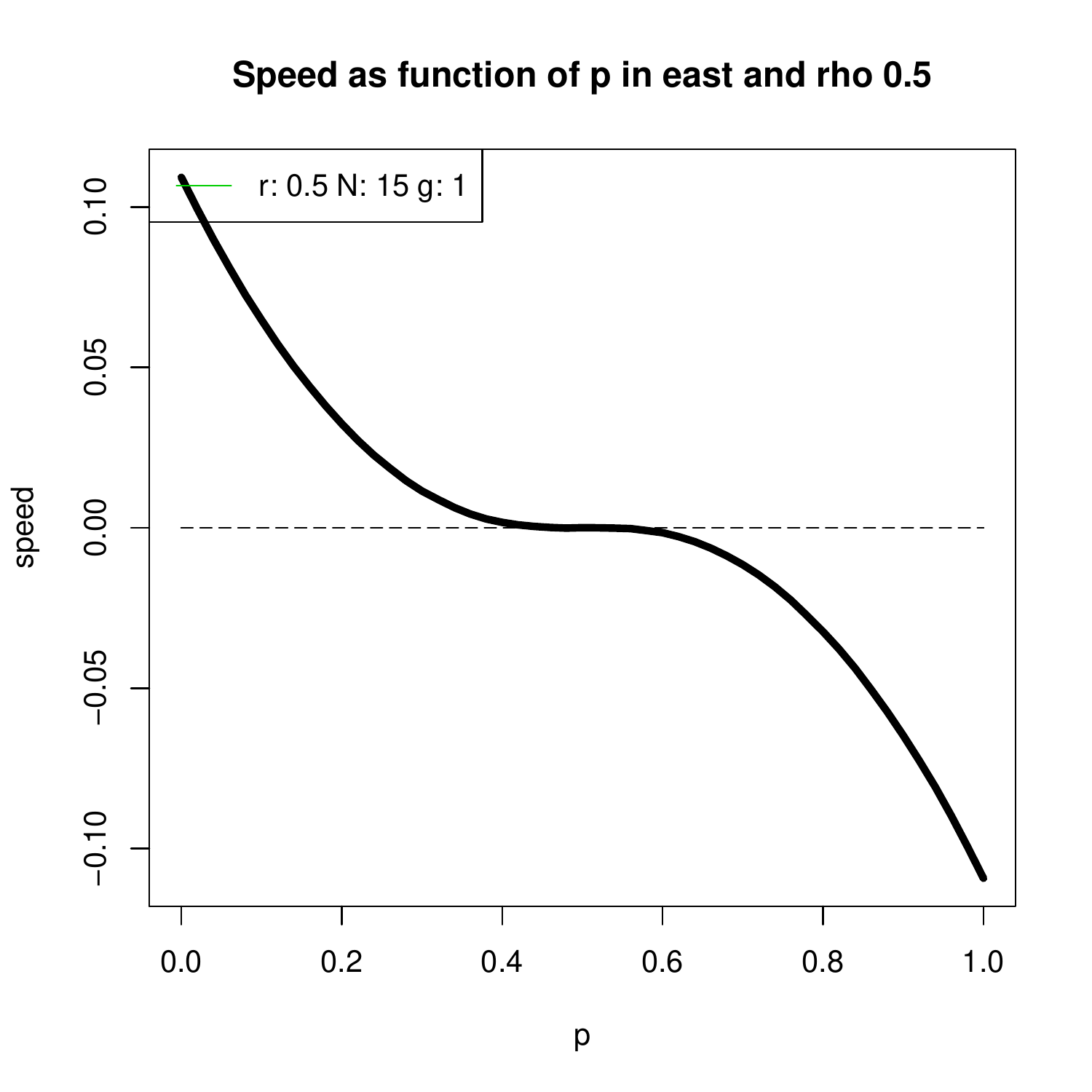}
\includegraphics[width=4cm]{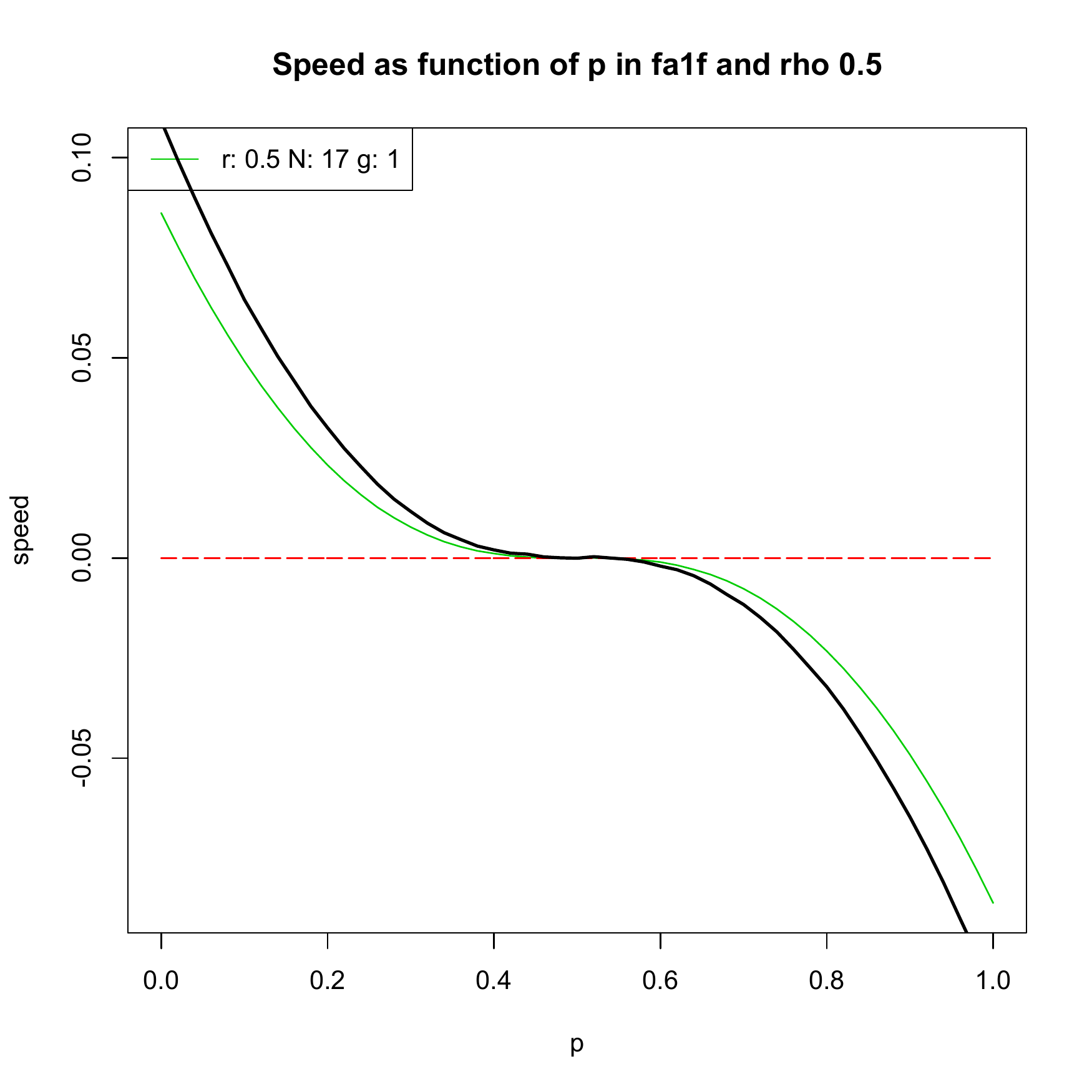} 
\end{center}
\caption{\small  Velocity as a function of $p:=1/2+\e\in [0,1]$ of the $\e$-RW in East (left picture, black curve) and FA1f (right picture, green curve) environments at density $\rho=1/2$. Note the antisymmetry $v(\e)=-v(-\e)$ as in Theorem \ref{guepe} and the non-zero velocity. The curves have been obtained by interpolation over points at distance $0.02$ in $[0,1]$, for each point, the corresponding value of the velocity is given by a sample-mean on $\approx 5000$ experiments in which the RW performed $2^{17}$ jumps. 
}
\label{pSpeed}
\end{figure}

It is tempting to interpret the sign of $v(\e)$ as a signature of the orientation of the East model and the asymmetry of its dynamics. However, as pointed out in Proposition~\ref{signature}, the antisymmetry relation \eqref{antisym} holds for the $\e$--RW on the East model. Recall the West model mentioned in Remark \ref{WestFA}, and denote by $P^{\rm east}_{\eta,0 }$ and $P^{\rm west}_{\eta,0 }$, the laws of the $\e$-RW in the environments East and West, respectively,   {starting at the origin with environment $\eta$}. Then, by    {considering a space reflection at the origin}, it is easy to see that $\forall A\subset \bbR$ and any $\eta\in\Omega$, $P^{\rm east}_{\eta,0 }(X^{(\e)}\in A)=P^{\rm west}_{\eta,0 }(-X^{(-\e)}\in A)$. Consequently, at any density, $v_{\rm east}(\e)=-v_{\rm west}(-\e)$, where $v_{\rm east}(\e)$ denotes the asympotic velocity in \eqref{velo} in the East environment, and similarly for West. In view of this observation, the following statement is a straightforward consequence of \eqref{antisym} and shows that the orientation of the environment does not determine the sign of the velocity.

\begin{Corollary}\label{EW}At any density $\rho\in (0,1)$
$$v_{\rm east}(\e)=v_{\rm west}(\e).$$
\end{Corollary} 

\begin{figure}
\begin{center}
\includegraphics[scale=.3]{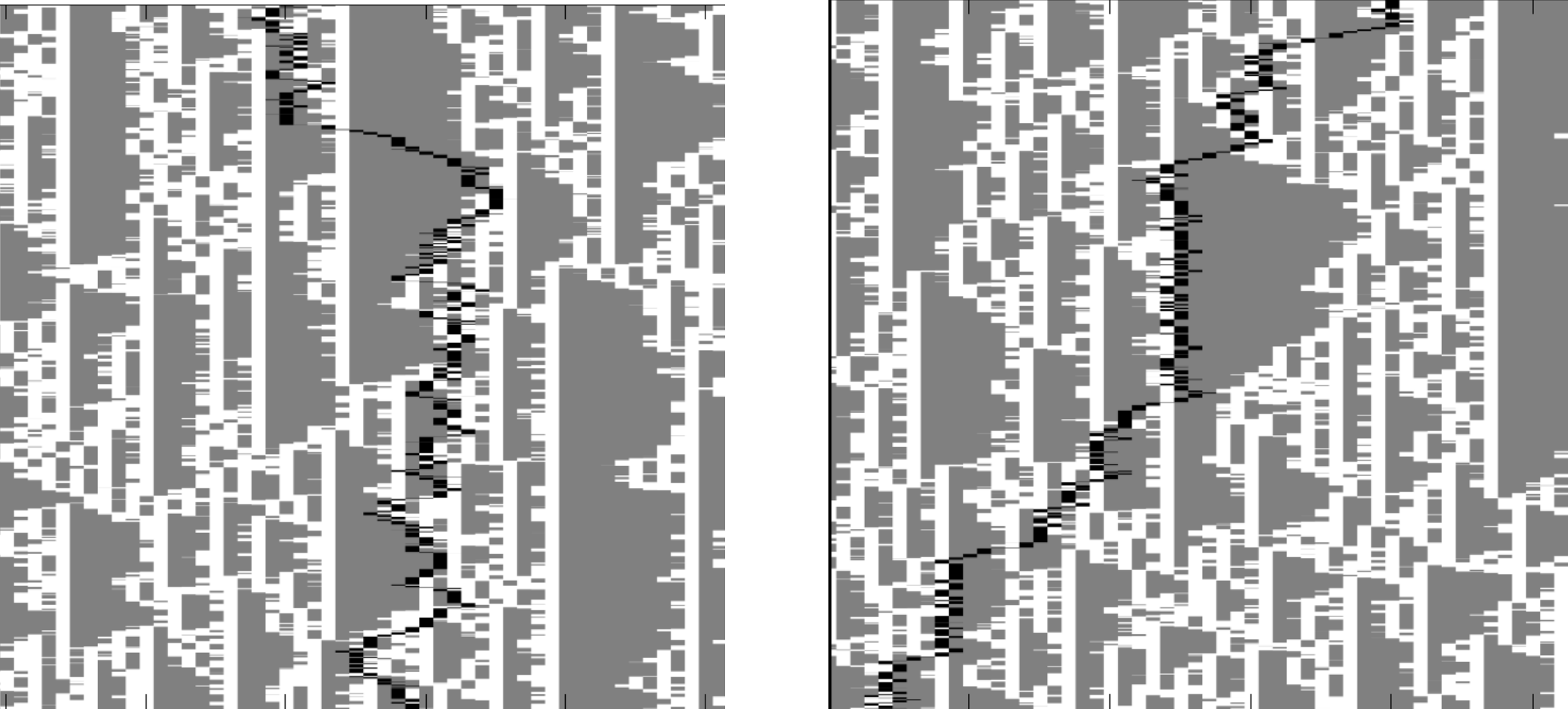}
\caption{Two simulations of $\e$-RW (in black) on the East model with $\rho=0.6$. Time goes down, space is horizontal; particles are in gray, holes in white. On the left, $\e=+0.3$. On the right, $\e=-0.3$. The two walkers have different-looking trajectories, even up to mirror reflection, due to the fact that one of them sticks to the fluctuating border of the bubbles and the other to the straight border.}
\label{twowalkers}
\end{center}
\end{figure}

Above we use the notation $v_{\rm east}(\e)$ and $v_{\rm west}(\e)$ to distinguish the velocities in two environments. From now on we consider only the East model and go back to the lighter notation $v(\e)$.

Besides the numerics in Figure~\ref{pSpeed}, we now show two different results supporting Conjecture~\ref{conj}. 
The following proposition provides a criterion in terms of space--time correlations of the environment implying the negativity of $\k$ defined in Proposition~\ref{signature} (recall that $v(\e)=\k \e^3+O(\e^5)$).

\begin{Proposition}\label{shaun}
Set $\rho=1/2$. If for all $s,t> 0$ and for all $y\geq 1$ it holds
\begin{equation}\label{criterion}
\bbE_{\nu_{1/2}}^{\mathrm{east}}\left[\xi_0(0)\left(2\xi_t(y)-1\right)\xi_{t+s}(0)\right]>0,
\end{equation}
then $\k$ in \eqref{k} is negative.
\end{Proposition}

Let us explain why we expect \eqref{criterion} to be true. It is clearly equivalent to the following inequality:
\[
\bbP_{\nu_{1/2}}^{\mathrm{east}}\left(\xi_0(0)=1,\xi_t(y)=1,\xi_{t+s}(0)=1\right)>\bbP_{\nu_{1/2}}^{\mathrm{east}}\left(\xi_0(0)=1,\xi_t(y)=0,\xi_{t+s}(0)=1\right)
\]
Since we are at density $1/2$ and due to the orientation of the East model (see Appendix {A}), we have $\bbP_{\nu_{1/2}}^{\mathrm{east}}\left(\xi_0(0)=1,\xi_t(y)=1\right)=\bbP_{\nu_{1/2}}^{\mathrm{east}}\left(\xi_0(0)=1,\xi_t(y)=0\right)=1/4$. Therefore the question is whether it is more likely to keep a particle initially present at time zero when $\xi_t(y)=0$ or $\xi_t(y)=1$. Intuitively, since zeros can send excitations that allow updates of particles to their left, having a particle at $y$ should work towards conserving a particle on its left, e.g. at the origin, which explains why \eqref{criterion} should hold.

The second argument supporting Conjecture~\ref{conj} is given by Proposition \ref{DegDrift}. More precisely, we introduce below a random walk which is a degenerate version of the $\e$--RW with $\e>0$ on the East model and show that it has negative velocity.
\subsubsection{A degenerate drifting RW}\label{4.4}
Let us introduce a degenerate version of the   {$\e$--RW}. Informally, we introduce a new random walk  $(Y_t)_{t \in \bbR_+}$ living on the edges of $\bbZ$ with a hole to the right and a particle to the left. Initially, the walker stands on the first edge on the right of the origin satisfying this condition. If the particle to its left flips into a hole, it jumps instantly to the next edge of this type to its left. If instead the hole to its right flips into a particle, the random walk jumps instantly to the next edge of this type to its right. See Figure~\ref{deg} for an illustration of the different possible jumps for the walker. The latter is a degenerate version of the $\e$--RW. Indeed it can be thought of as jumping at infinite rate to the left (resp.\@ right) when it is sitting on top of a hole (resp.\@ particle).

\begin{center}
\begin{figure}
\includegraphics[scale=.5]{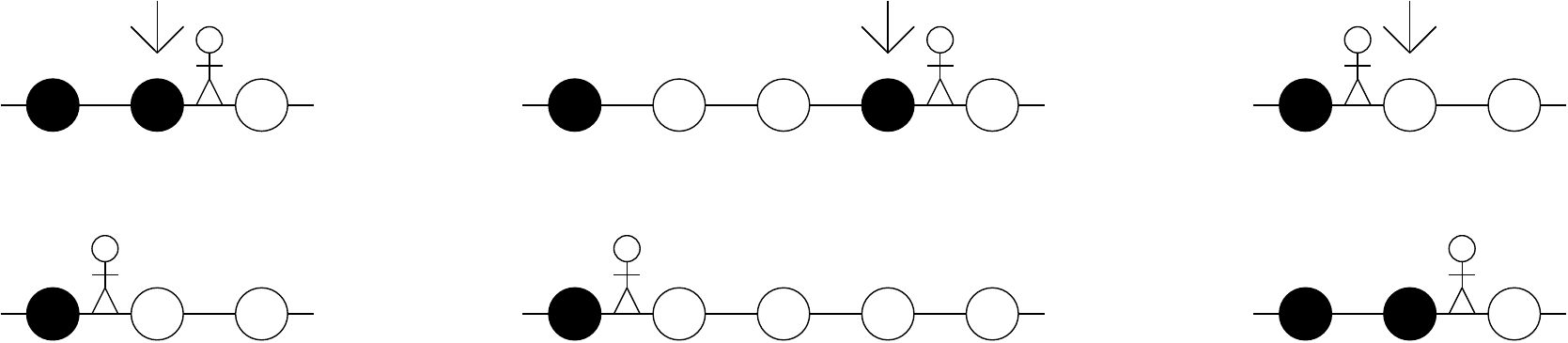}
\caption{Illustration of the different updates in the environment inducing a jump of the degenerate random walker. Particles and holes of the environment are represented by black and white disks respectively. On top, three pre-update situations are depicted, with an arrow pointing the position of next flip occuring in the dynamics of the environment. On the bottom, one can see the result of the spin flip, in particular the induced new position of the walker.}
\label{deg}
\end{figure}
\end{center}

Let us give the precise definition of the joint process $(\xi_t, Y_t)_{t \geq 0}$. To this aim, we parametrize the edges of $\bbZ$ by $1/2+\bbZ$ assigning to each $y \in 1/2+\bbZ$ the edge $\{y^-, y^+\}$, where  $y^+:=y+1/2$ and $y^-:=y-1/2$.  The initial configuration $(\xi_0, Y_0)$ is determined as follows: 
$\xi_0\in\Omega$ is sampled from $\nu_{\rho}$ while  $Y_0:=k+1/2$, where $k$ is the lowest non-negative integer such that   {$\xi_0(k)=1$}, $\xi_0(k+1)=0$. Then the Markov generator of the joint process $(\xi_t,Y_t)$ is given by
\begin{eqnarray}
 L_{deg}f(\xi,y)=\displaystyle\sum_{x\notin{\lbrace y^-,y^+\rbrace}}c^{\rm east}_x(\xi)  \left[\rho(1-\eta(x))+(1-\rho)\eta(x)\right]\left[f(\eta^x,y)-f(\eta,y)\right]\nonumber\\
\qquad\displaystyle +\,(1-\rho) \left[f(\xi^{y^-}, y- k (\xi, y_-))-f(\xi,y)\right]\nonumber\\
\qquad+\,\rho\, c^{\rm east}_{y_+}(\xi)   \left[f(\xi^{y^+},y+1)-f(\xi,y)\right],
\end{eqnarray}
where $ k(\xi, y_-)$ is the first positive integer    {$k$} such that $\xi(y^--k)=1$. 


The series in the r.h.s. corresponds to updates of the environment occuring on sites not belonging to the edge where the walker sits. The second line describes what happens when the particle on the left of the walker disappears, and the third when a particle appears on the right of the walker. 

Notice that the above joint process is well defined, since at any time there are infinitely many sequences of particle--hole when the initial configuration is sampled from $\nu_{\rho}$ with density $\rho\in (0,1)$.  Moreover, the evolution of $\xi_t$ is the standard East model at equilibrium.

At any density $\rho$, the random walk $(Y_t)_{t\geq 0}$ has a negative velocity:
\begin{Proposition}\label{DegDrift}
For any $\rho\in(0,1)$ it holds $\limsup_{t\rightarrow \infty} \frac{Y_t}{t}<0 $ a.s.
\end{Proposition}
  {The proof of the above proposition  is given in Section \ref{proofDegDrift}.}
 
\subsection{Density profile}
By applying  Proposition \ref{denso} to specific environments a more detailed description of the observed density profile can be derived. In the case of the East model, we get the following.

 \begin{Corollary}\label{oscar} Suppose that $2|\e|$ is smaller than the spectral gap of $\mathcal{L}_{\rm east}$. Then 
 for any $x\in\bbZ$, it holds \begin{equation}\label{limitdensity}
\mu_\e(\eta(x))=\rho+2\e\int_0^\infty  u(s)\left[p_s(x-1)-p_s(x+1)\right]ds+O(\e^2),
\end{equation}
where $$u(s):=\rho^2-\nu_{\rho}\big(\xi(0)\bbE^{\rm east}_\xi[\xi_s(0)]\big)$$ is a negative increasing function on $[0,+\infty)$, with $u(0)=-\rho(1-\rho)$ and $|u(s)|\leq \rho(1-\rho)^{1/2}e^{-\l s}$. 

In particular, if $x<0$ (resp.\@ $x>0$), for $\e>0$ small enough, 
\begin{equation}\label{couledesource}
\mu_\e(\eta(x))>\rho\quad(\text{resp. }<\rho).
\end{equation}
Moreover, for any $x,y\in\mathbb Z$, we have that 
\begin{equation}\label{difference}
\begin{aligned}
&\mu_\e(\eta(x))-\mu_\e(\eta(x+y))=\e\int_0^\infty  u(s)[p_s(x-1)-p_s(x+1)\\
&-p_s(x+y-1)+p_s(x+y+1) ]ds+O(\e^2).
\end{aligned}
\end{equation}
 \end{Corollary} 
 
 \begin{remark}[{\bf Equilibrium affected around the origin}]
From Theorem 3 in \cite{mainprobe}, we know that, as $|x|\rightarrow\infty$, the density profile of the environment seen by the walker, $(\mu_\e(\eta(x)))_{x\in\bbZ}$, approaches the constant profile $\rho$ (corresponding to the equilibrium of the environment process). This reflects the fact that the $\e$-RW is ``sitting at the origin" of the process $(\eta^{(\e)}_t)_{t\geq 0}$. 
Moreover, recall that for $\e>0$ the random walk has a tendency to jump to the right when sitting on top of particles, and vice versa on top of holes. Therefore, heuristically, it should spend more time with a particle to its left and a hole to its right. \eqref{couledesource} and \eqref{difference} confirm this description for $\e$ small enough. Indeed, from \eqref{difference}, it is simple to see that $\mu_\e(\eta(-1))>\mu_\e(\eta(0))>\mu_\e(\eta(1))$ for small $\e>0$.
These observations are summarized in the numerics in Figure \ref{dens0.5}.
\end{remark}
 
\begin{figure}
\begin{center}
\includegraphics[scale=.3]{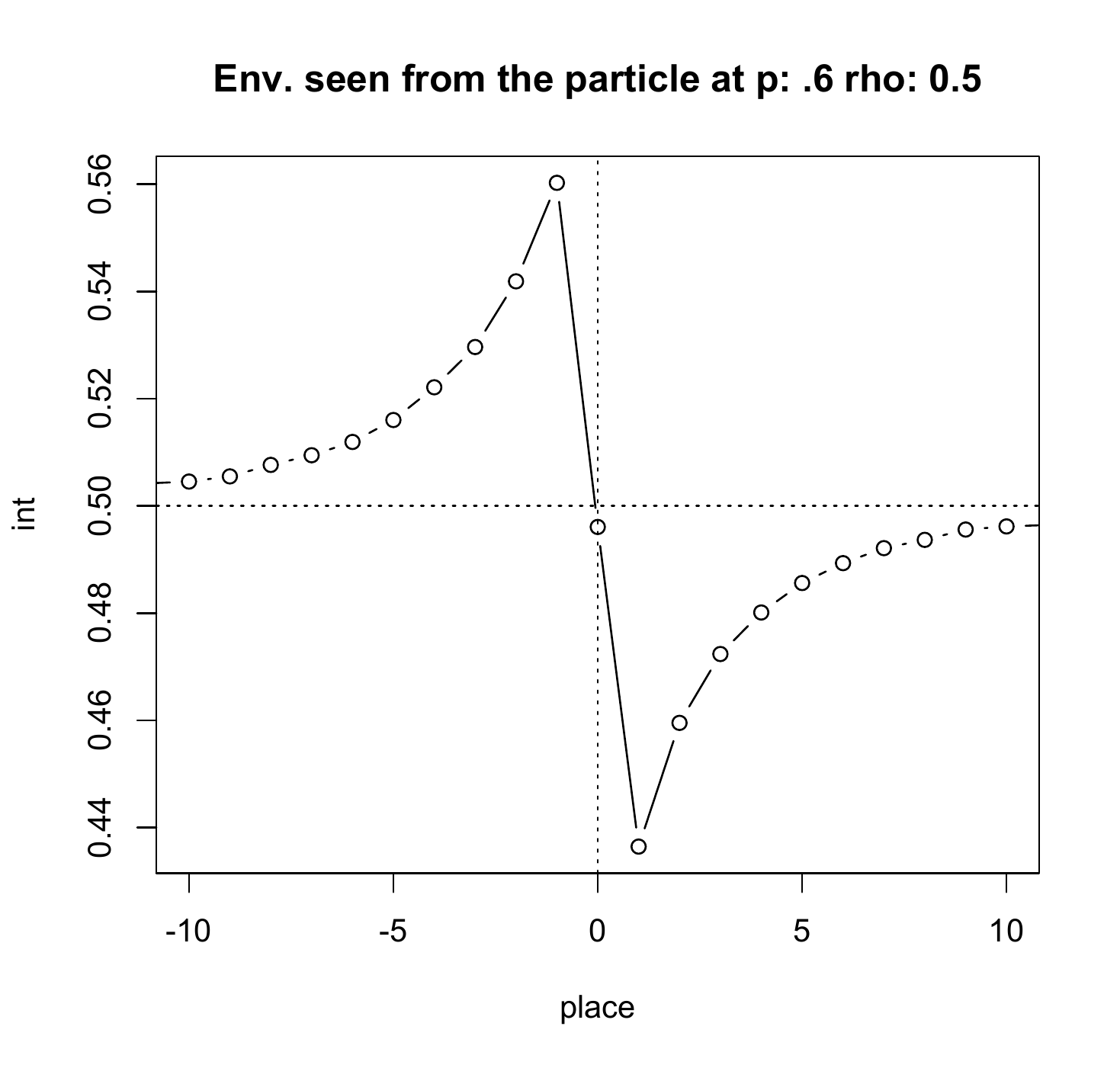}
\caption{Simulation of the density profile $\mu_\e(\eta(x))$, $x=-10,\ldots,10$ at $\rho=1/2$ and $\e=0.1$.}
\label{dens0.5}
\end{center}
\end{figure}

\section{Proofs}\label{6}

\subsection{Proof of Theorem~\ref{guepe}}\label{proofguepe}
  {For the sake of clarity, we give the proof  for the  $\e$--RW. The proof  can be easily generalized to the transition rates satisfying Assumption \ref{brioche}.}

We first recall a lemma from \cite{mainprobe}, Lemma 10.1 therein. The notation is adapted to this specific context.
\begin{Lemma}\label{nasty}
Under Assumptions~\ref{piqure}, $t\geq 0$, $x\in\bbZ$, then for all $n\geq 0$,  $f: \O \to \bbR$ local function and  $\eta\in\O$, 
it holds
\begin{equation}\label{pingu}
\begin{split}
\hat{L}_\e S_\e^{(n)}(s)f(\eta)= &  \int_0^tdt_1\int_0^{t_1}dt_2\ldots\int_0^{t_{n-1}}dt_n\sum_{{\mathbf{z}\in \{\pm 1\}^{n+1},}}\sum_{\d\in\lbrace 0,1\rbrace^{n+1}}(-1)^{|\d|}\\
& \times \bbE_\eta^{(0)}\left[\left(\prod_{i=1}^{n+1}\hat{r}_\e\left(z_i,\tau_{(\d\cdot z)_{[i-1]}}\eta_{t-t_{i-1}}\right)\right)f\left(\tau_{(\mathbf{\d\cdot \mathbf{z}})_{[n+1]}}\eta_{t}\right)\right]\,,\end{split}
\end{equation}
where $\hat{r}_\e(\pm 1,\eta):=\pm\e(2\eta(0)-1)$, $|\d|:=\sum_{i=1}^{n+1}(1-\d_i)$, $(\d\cdot \mathbf{z})_{[i]}:=\d_1z_1+\ldots+\d_i z_i$, $(\d\cdot \mathbf{z})_{[0]}:=0$ and $t_0:=t$.
\end{Lemma}
Formula \eqref{pingu} has to be thought with  no time integration in the  degenerate case $n=0$. The central ingredient in the proof of Theorem~\ref{guepe} is the following lemma.

\begin{Lemma}\label{fight}
Assume Assumptions \ref{piqure}.
Then, for all $n\geq0$, 
we have 
\begin{equation}\label{second}
\mu\left(\hat{L}_\e S_\e^{(2n)}(t)j^{(\e)}(\eta)\right)=0.
\end{equation}
\end{Lemma}

\begin{proof}\label{nastybis}

For all $n\geq 1$, for all $\eta\in\O$,  by applying Lemma~\ref{nasty} to the function $j^{(\e)}(\eta)=2\e (2\eta(0)-1)=\pm 2 \hat{r}_\e(\pm 1,\eta)$, we can write

\begin{eqnarray}
\hat{L}_\e S_\e^{(n)}(t)j^{(\e)}(\eta)&=&\displaystyle\int_0^tdt_1\int_0^{t_1}dt_2\ldots\int_0^{t_{n-1}}dt_{n}\sum_{\mathbf{\d}\in\lbrace 0,1\rbrace^{n+1}}(-1)^{|\mathbf{\d}|}\sum_{{\mathbf{z}\in \{\pm 1\}^{n+1}}}\nonumber\\\notag
&&\displaystyle\,\bbE_\eta^{(0)}\left[\left(\prod_{i=1}^{n+1}\e z_if\left(\tau_{(\mathbf{\d\cdot z})_{[i-1]}}\eta_{t-t_{i-1}}\right)\right)2\e f\left(\tau_{(\mathbf{\d\cdot z})_{[n+1]}}\eta_{t}\right)\right],\\
&=&2\e^{n+2}\int_0^tdt_1\int_0^{t_1}dt_2\ldots\int_0^{t_{n-1}}dt_{n}\sum_{\mathbf{\d}\in\lbrace 0,1\rbrace^{n+1}}\sum_{{\mathbf{z}\in \{\pm 1\}^{n+1}}}\nonumber\\\label{cara}
&&\displaystyle\,(-1)^{|\d|}\left(\prod_{i=1}^{n+1}z_i\right)\bbE_\eta^{(0)}\left[\left(\prod_{i=1}^{n+2}f\left(\tau_{(\mathbf{\d\cdot z})_{[i-1]}}\eta_{t-t_{i-1}}\right)\right)\right].
\end{eqnarray}

with $f(\eta):=(2\eta(0)-1)$ and $t_{n+1}:=0, t_0:=t$.

To shorten the notation, write $A_t:=\lbrace \mathbf{t}\in\bbR^{n} , 0\leq t_{n}\leq\ldots\leq t_1\leq t\rbrace$. For $\mathbf{t}\in A_t$, $\mathbf{\d}\in\lbrace 0,1\rbrace^{n+1}$, $\mathbf{z}\in \{\pm 1\}^{n+1}$, let
\begin{equation}\label{phi}
\phi_n(\mathbf{t},\mathbf{\d},\mathbf{z}):=\bbE_\mu^{(0)}\left[\left(\prod_{i=1}^{n+2}f\left(\tau_{(\mathbf{\d\cdot z})_{[i-1]}}\eta_{t-t_{i-1}}\right)\right)\right].
\end{equation}

Then, by \eqref{cara} and \eqref{phi}, for any $n\geq 1$, we can write that

\begin{align}\label{cru}
\mu\left(\hat{L}_\e S_\e^{(n)}(t)j^{(\e)}(\eta)\right)
=2\e^{n+2}\sum_{\mathbf{\d}\in\lbrace 0,1\rbrace^{n+1}}
(-1)^{|\d|}\sum_{{\mathbf{z}\in \{-1,+1\}^{n+1}}}
\left(\prod_{i=1}^{n+1}z_i\right)\int_{A_t}\mathrm{d}\mathbf{t}\phi_n(\mathbf{t},\mathbf{\d},\mathbf{z}),
\end{align}

We are now going to show that, for any $n\geq 1$, \eqref{cru} reduces to 
\begin{align}\label{first}
\mu\left(\hat{L}_\e S_\e^{(n)}(t)j^{(\e)}(\eta)\right)
=2\e^{n+2}(-1)^{n+1}\sum_{{\mathbf{z}\in \{-1,+1\}^{n+1}}}
\left(\prod_{i=1}^{n+1}z_i\right)\int_{A_t}\mathrm{d}\mathbf{t}\phi_n(\mathbf{t},\mathbf{1},\mathbf{z}),
\end{align}

Take $\d\neq \mathbf{1}$ and let $j\in\lbrace 1,\ldots, n+1\rbrace$ such that $\d_j=0$. 
Then, for every such $\mathbf{\d}\in\lbrace 0,1\rbrace^{n+1}$, we get

\begin{eqnarray}
\sum_{{\mathbf{z}\in \{\pm 1\}^{n+1}}}\left(\prod_{i=1}^{n+1}z_i\right)\phi_n(\mathbf{t},\mathbf{\d},\mathbf{z})&=&\sum_{\underset{z_j=+1}{\mathbf{z}\in \{\pm 1\}^{n+1}}}\left(\prod_{\underset{i\neq j}{i=1,\ldots, n+1}}z_i\right)\phi_n(\mathbf{t},\mathbf{\d},\mathbf{z})\\
&& -\,\sum_{\underset{z_j=-1}{\mathbf{z}\in \{\pm 1\}^{n+1}}}\left(\prod_{\underset{i\neq j}{i=1,\ldots, n+1}}z_i\right)\phi_n(\mathbf{t},\mathbf{\d},\mathbf{z})\nonumber=0,
\end{eqnarray}

where in the last equality we have used that $\phi_n(\mathbf{t},\mathbf{\d},\mathbf{z})$ does not depend on $z_j$ since $\d_j=0$.
Hence, \eqref{first} is proven.

In the next steps we will first use reversibility and then that $n$ is even.

On the one hand, by the change of variable, \[ A_t \ni \mathbf{t}=(t_1, t_2, \dots, t_{n}) \mapsto \mathbf{t}^*:=( t-t_{n}, t-t_{n-1}, \dots, t-t_1) \in A_t,\]
we have 
\begin{equation}\label{changeVar}
\int_{A_t}\mathrm{d}\mathbf{t}\ \phi_n(\mathbf{t},\mathbf{1},\mathbf{z})=\int_{A_t}\mathrm{d}\mathbf{t}\ \phi_n(\mathbf{t}^*,\mathbf{1},\mathbf{z}).
\end{equation}

On the other hand, by reversibility, we can show that 
\begin{equation}\label{revStep}
\phi_n(\mathbf{t},\mathbf{1},\mathbf{z})=\phi_n(\mathbf{t}^*,\mathbf{1},\overline{\mathbf{z}}),
\end{equation}
where $\overline{\mathbf{z}}:=-(z_{n+1},\ldots ,z_1)$.

Let us check \eqref{revStep}. Write first 
\begin{align}\label{abc100}\notag
\phi_n(\mathbf{t},\mathbf{1},\mathbf{z})&=\bbE_\mu^{(0)}\left[\left(\prod_{i=1}^{n+2}f\left(\tau_{(\mathbf{1\cdot z})_{[i-1]}}\eta_{t-t_{i-1}}\right)\right)\right]= \bbE_\mu^{(0)}\left[\left(\prod_{i=0}^{n+1}f\left(\tau_{(\mathbf{1\cdot z})_{[i]}}\eta_{t-t_{i}}\right)\right)\right]
\\\notag
&=\bbE_\mu^{(0)}\left[\left(\prod_{i=0}^{n+1}f\left(\tau_{(\mathbf{1\cdot z})_{[i]}}\eta_{t_{i}}\right)\right)\right]
=\bbE_\mu^{(0)}\left[\left(\prod_{l=0}^{n+1}f\left(\tau_{(\mathbf{1\cdot z})_{[n+1-l]}}\eta_{t_{n+1-l}}\right)\right)\right]
\\
&=\bbE_\mu^{(0)}\left[\left(\prod_{l=0}^{n+1}f\left(\tau_{(\mathbf{1\cdot z})_{[n+1-l]}}\eta_{t-t^*_l}\right)\right)\right],
\end{align}
where the third identity follows by reversibility, and the last one by the mapping $\mathbf{t}\mapsto \mathbf{t}^*$ since $t_l^*=t- t_{n+1-l}$.

Next, note that 
\begin{equation}\notag
(\mathbf{\mathbf{1}\cdot  z})_{[n+1-l]} = z_1+ \dots z_{n+1-l} =- ( \bar z_{n+1}+ \bar z_{n}+ \dots+ \bar z_{l+1})= b +
(\mathbf{\mathbf{1}\cdot  \bar z})_{[l]},
\end{equation}
where $b:= - \sum_{j=1}^{n+1}\bar z_j$.
Hence, from \eqref{abc100} and translation invariance,  we get

\begin{align*}\notag
\phi_n(\mathbf{t},\mathbf{1},\mathbf{z})&
=\bbE_\mu^{(0)}\left[\left(\prod_{l=0}^{n+1}f\left(\tau_{(\mathbf{1\cdot z})_{[n+1-l]}}\eta_{t-t^*_l}\right)\right)\right]
=\bbE_\mu^{(0)}\left[\left(\prod_{l=0}^{n+1}f\left(\tau_b\tau_{(\mathbf{\mathbf{1}\cdot \bar z})_{[l]}}\eta_{t-t^*_l}\right)\right)\right]
\\&
=\bbE_\mu^{(0)}\left[\left(\prod_{l=0}^{n+1}f\left(\tau_{(\mathbf{\mathbf{1}\cdot \bar z})_{[l]}}\eta_{t-t^*_l}\right)\right)\right]
=\phi_n(\mathbf{t}^*,\mathbf{1},\overline{\mathbf{z}}),
\end{align*}
as claimed in \eqref{revStep}.

Finally, for any $n\geq 0$, due to \eqref{changeVar} and \eqref{first}, we can write
 
\begin{align*}
\mu\left(\hat{L}_\e S_\e^{(2n)}(t)j^{(\e)}(\eta)\right)=& 2\e^{2n+2}(-1)^{2n+1}\sum_{{\mathbf{z}\in \{\pm 1\}^{2n+1}}}
\left(\prod_{i=1}^{2n+1}z_i\right)\int_{A_t}\mathrm{d}\mathbf{t}\phi_{2n}(\mathbf{t},\mathbf{1},\mathbf{z})\\
&=-\e^{2n+2}\sum_{{\mathbf{z}\in \{\pm 1\}^{2n+1}}}
\left(\prod_{i=1}^{2n+1}z_i\right)\int_{A_t}\mathrm{d}\mathbf{t}
\left[\phi_{2n}(\mathbf{t},\mathbf{1},\mathbf{z})+ 
\phi_{2n}(\mathbf{t}^*,\mathbf{1},\mathbf{z})\right]
\end{align*}

Therefore, to get the claim in \eqref{second}, it suffices to show that, for any $\mathbf{t}\in A_t$,
\begin{equation}\label{end}
\sum_{{\mathbf{z}\in \{\pm 1\}^{2n+1}}}
\left(\prod_{i=1}^{2n+1}z_i\right)\left[\phi_{2n}(\mathbf{t},\mathbf{1},\mathbf{z})+ 
\phi_{2n}(\mathbf{t}^*,\mathbf{1},\mathbf{z})\right]=0.
\end{equation}

In fact, by using that for $2n$ even, $\prod_{i=1}^{2n+1}\bar{z}_i=-\prod_{i=1}^{2n+1}z_i$, we can write 

\begin{align}
&\sum_{{\mathbf{z}\in \{\pm 1\}^{2n+1}}}
\left(\prod_{i=1}^{2n+1}z_i\right)\left[\phi_{2n}(\mathbf{t},\mathbf{1},\mathbf{z})+ 
\phi_{2n}(\mathbf{t}^*,\mathbf{1},\mathbf{z})\right]
=\sum_{{\mathbf{z}\in \{\pm 1\}^{2n+1}}}\left(\prod_{i=1}^{2n+1}z_i\right)\phi_{2n}(\mathbf{t},\mathbf{1},\mathbf{z})\\&+
\sum_{{\mathbf{z}\in \{\pm 1\}^{2n+1}}}\left(\prod_{i=1}^{2n+1}\bar{z}_i\right)\phi_{2n}(\mathbf{t}^*,\mathbf{1},\mathbf{\bar z})=
\sum_{{\mathbf{z}\in \{\pm 1\}^{2n+1}}}\left(\prod_{i=1}^{2n+1}z_i\right)
[\phi_{2n}(\mathbf{t},\mathbf{1},\mathbf{z})-\phi_{2n}(\mathbf{t}^*,\mathbf{1},\mathbf{ \bar z})],
\end{align}
and as claimed in \eqref{end}, the latter equals to zero  due to \eqref{revStep}.
\end{proof}

We are now in shape to conclude the proof of Theorem \ref{guepe}. 
Lemma \ref{fight} implies the cancellation of the even terms in the expansion of the velocity \eqref{Expvelo}, that is, 
\begin{equation}\label{anti}
v(\e)=2\e(2\mu(\eta(0))-1)+\sum_{n=0}^\infty\int_0^\infty \mu(\hat L_\e S_\e^{(2n+1)}(s)j^{(\e)})ds,
\end{equation}
and from equation \eqref{anti} the claim readily follows.
Indeed, by using \eqref{first} in Lemma \ref{fight}, for any $n\geq 0$, we have that : 
\begin{align*}
&\int_0^\infty ds\mu(\hat L_\e S_\e^{(2n+1)}(s)j^{(\e)})=2\e^{2n+1}\sum_{\mathbf{z}\in \{\pm 1\}^{2n}}
\left(\prod_{i=1}^{2n}z_i\right)\times\\&\times\int_0^\infty ds \int_{A_s}\mathrm{d}\mathbf{t}
\bbE_\mu^{(0)}\left[\left(\prod_{i=1}^{2n+1}\left(2\eta_{t-t_{i-1}}(\sum_{l=1}^{i-1}z_l)-1\right)\right)\right]=:
\e^{2n+1}c_{2n}.
\end{align*} 
By plugging the above expression into \eqref{anti}, we get 
$$v(\e)=2\e(2\mu(\eta(0))-1)+\e\sum_{n=1}^\infty\e^{2n}c_{2n},$$  
which, as claimed, is an antisymmetric function of $\e$.

\subsection{Proof of Proposition~\ref{signature}}\label{proofsignature}
By Lemma \ref{fight} and \eqref{anti}, it suffices to show that 
\begin{equation}\label{end2}\int_0^\infty \mu(\hat L_\e S_\e^{(1)}(s)j^{(\e)})ds=\e^3\kappa,\end{equation}
with $\kappa$ as in \eqref{k}.
To shorten the computations, we abbreviate $f(\eta):=2\eta(0)-1$ and $h(\eta):=2[\eta(1)-\eta(-1)]$.
Compute first
\begin{equation*}
\begin{split} 
\mu(\hat L_\e S_\e^{(1)}(s)f) &
=\e\mu(f S_\e^{(1)}(s)h)
=-\e\mu\left(h\,S_\e^{(1)}(s)f\right)
=-\e\int_0^sdu\,\mu\left(h    {\cdot  [S(s-u)\hat L_\e S(u)f]}\right)\\
&=-\e\int_0^sdu\,\mu\left([S(s-u)h]\,[\hat L_\e S(u)f]\right)
= -\e^2\int_0^sdu\,\mu\left(  { [ S(s-u)h] f [ S(u)h ]}\right),
\end{split}
\end{equation*}
where  we have used the definition of $\hat L_\e$ in   {\eqref{michele}},  translation invariance, 
 the  definition of $S_\e^{(1)}(s)$, reversibility    {and Assumption~\ref{piqure}--(iii)}.

Hence
\begin{eqnarray}\notag 
&\int_0^\infty\mu(\hat L_\e S_\e^{(1)}(s)f)ds=-\e^2\int_0^\infty ds\int_0^sdu\,\mu\left(f   {[S(s-u)h] [S(u)h]}\right)\\&\notag
=- \e^2  \int_0^\infty du\int_0^\infty dz\, \mu\left(f\,S(z)h\,S(u)h\right)
=-  \e^2\mu\Big(f\Big[\int_0^\infty ds\,S(s)h\Big]^2\Big)\\&
=- 4 \e^2\mu\Big(f\Big[\int_0^\infty ds\,\bbE_{\eta}^{(0)}[\eta_1(s)-\eta_{-1}(s)]\Big]^2\Big)\label{expressionkappa} .
\end{eqnarray}
 Finally, \eqref{end2} follows by recalling that $j^{(\e)}(\eta)=2\e f(\eta)$ and \eqref{expressionkappa} above.

\subsection{Proof of Proposition~\ref{shaun}}\label{proofshaun}
From \eqref{expressionkappa}, we have
\begin{equation}
\k=-2\int_0^\infty dt\int_0^\infty ds\ \nu_\rho\left(h\cdot S(t)\left(fS(s)h\right)\right).
\end{equation}
Recall that $p_t(y)$ is the probability that a continuous time SRW started from $0$ is at $y$ at time $t$. Using that $S(t)g(\eta)=\sum_{y\in\bbZ}p_t(y)\bbE_\eta^{\mathrm{east}}\left[g(\tau_y\eta_t)\right]$ and the Markov property, we can rewrite the term in the double integral in the above expression as $A-B-C+D$, where
\begin{eqnarray}
A&=&\sum_{y,z\in\bbZ}p_t(y)p_s(z)\bbE_{\nu_\rho}^{\mathrm{east}}\left[\eta(1)\left(2\eta_t(y)-1\right)\eta_{t+s}(z+y+1)\right]\\
B&=&\sum_{y,z\in\bbZ}p_t(y)p_s(z)\bbE_{\nu_\rho}^{\mathrm{east}}\left[\eta(-1)\left(2\eta_t(y)-1\right)\eta_{t+s}(z+y+1)\right]\\
C&=&\sum_{y,z\in\bbZ}p_t(y)p_s(z)\bbE_{\nu_\rho}^{\mathrm{east}}\left[\eta(1)\left(2\eta_t(y)-1\right)\eta_{t+s}(z+y-1)\right]\\
D&=&\sum_{y,z\in\bbZ}p_t(y)p_s(z)\bbE_{\nu_\rho}^{\mathrm{east}}\left[\eta(-1)\left(2\eta_t(y)-1\right)\eta_{t+s}(z+y-1)\right].
\end{eqnarray}
Lemma \ref{orientation} shows that the expectations appearing in $A$ (resp. $B$, resp. $C$, resp. $D$) cancel as soon as $1,y,z+y+1 $ (resp. $-1,y,z+y+1$, resp. $1,y,z+y-1$, resp. $-1,y,z+y-1$) are pairwise distinct. The same holds in the cases where $y<1$ and $z+y+1=1$ (resp. $y<-1$ and $z+y+1=-1$, etc). In fact we only have to deal with terms where the three space-time points involved in the expectation are in one of the six schematic configurations of Figure~\ref{six}.
\begin{figure}
\begin{center}
\includegraphics[scale=.4]{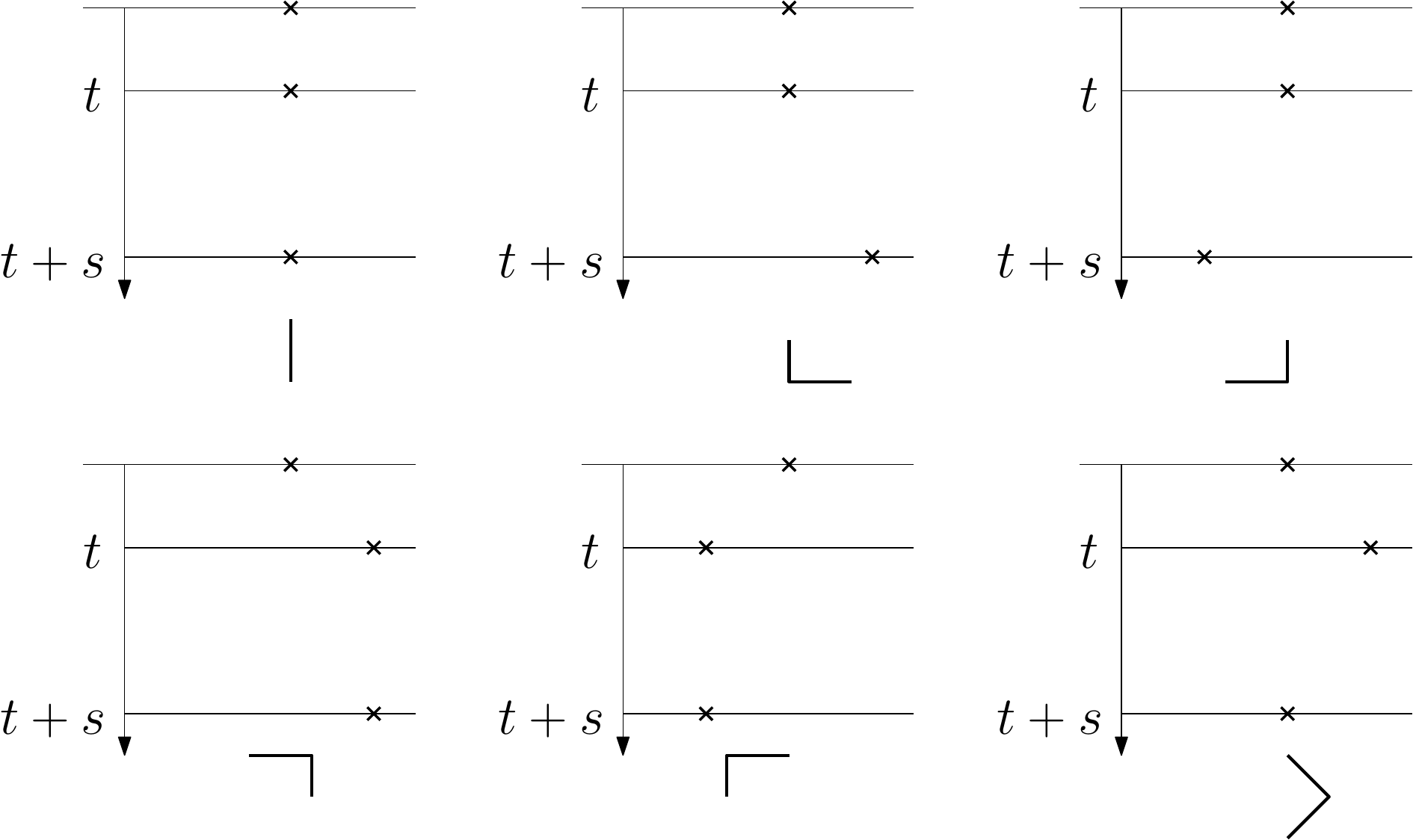} 
\caption{The only non-zero terms in $A$ are those where $(1,0),(y,t)$ and $(y+z+1,t+s)$ are in one of the six above respective positions (marked by crosses). Similarly for the non-zero terms in $B,C,D$.}
\label{six}
\end{center}
\end{figure}
We gather the terms corresponding to the different positions in $A_|,A_\urcorner,A_\ulcorner,A_\llcorner,A_\lrcorner,A_\rangle$ (and similarly for terms in $B,C,D$) so that
\begin{eqnarray}
A_|&=&p_t(1)p_s(-1)\bbE_{\nu_\rho}^{\mathrm{east}}\left[\eta(1)\left(2\eta_t(1)-1\right)\eta_{t+s}(1)\right],\\
A_\urcorner &=& \sum_{z\geq 0}p_t(1)p_s(z)\bbE_{\nu_\rho}^{\mathrm{east}}\left[\eta(1)\left(2\eta_t(1)-1\right)\eta_{t+s}(2+z)\right],\\
A_\ulcorner &=& \sum_{z\leq -2}p_t(1)p_s(z)\bbE_{\nu_\rho}^{\mathrm{east}}\left[\eta(1)\left(2\eta_t(1)-1\right)\eta_{t+s}(2+z)\right],\\
A_\llcorner &=&\sum_{y\geq 2}p_t(y)p_s(-1)\bbE_{\nu_\rho}^{\mathrm{east}}\left[\eta(1)\left(2\eta_t(y)-1\right)\eta_{t+s}(y)\right],\\
A_\lrcorner &=&\sum_{y\leq 0}p_t(y)p_s(-1)\bbE_{\nu_\rho}^{\mathrm{east}}\left[\eta(1)\left(2\eta_t(y)-1\right)\eta_{t+s}(y)\right],\\
A_\rangle &=&\sum_{y\geq 2}p_t(y)p_s(-y)\bbE_{\nu_\rho}^{\mathrm{east}}\left[\eta(1)\left(2\eta_t(y)-1\right)\eta_{t+s}(1)\right].\\
\end{eqnarray}
Translation invariance and the symmetry of the heat kernel imply that
\begin{eqnarray}
A_|=B_|=C_|=D_|,\\
A_\urcorner=B_\urcorner, C_\urcorner=D_\urcorner,\\
A_\ulcorner=B_\ulcorner, C_\ulcorner=D_\ulcorner,\\
A_\llcorner+A_\lrcorner=B_\llcorner+B_\lrcorner, C_\llcorner+C_\lrcorner=D_\llcorner+D_\lrcorner,
\end{eqnarray}
so that many terms cancel and the sum reduces to
\begin{equation}
A-B-C+D=A_\rangle-B_\rangle-C_\rangle+D_\rangle.
\end{equation}
Rearranging this expression by using translation invariance, we get that $A-B-C+D$ is nothing but
\begin{equation}
\sum_{y\geq 1}\left(p_t(y+1)-p_t(y-1)\right)\left(p_s(y+1)-p_s(y-1)\right)\bbE_{\nu_\rho}^{\mathrm{east}}\left[\eta(0)\left(2\eta_t(y)-1\right)\eta_{t+s}(0)\right].
\end{equation}
The claim follows by noticing that for $s,t>0$, $y\geq 1$,
\begin{equation}
\left(p_t(y+1)-p_t(y-1)\right)\left(p_s(y+1)-p_s(y-1)\right)>0.
\end{equation}

\subsection{Proof of Proposition~\ref{DegDrift}}\label{proofDegDrift}
The result is a consequence of Lemma 3.2 in \cite{O2}. Let us just recall which process is considered in \cite{O2} and show how it can be coupled with our walker.   {We refer to Appendix \ref{A} for some standard terminology concerning the East model}.

In \cite{O2}, the author studies the evolution of a single hole, which is called \emph{the front}, through the East dynamics. 
The ``front process'' $(F_t)_{t\geq 0}$ is constructed as follows.
Start with any configuration $\eta\in\O$ with a hole at site $k$, i.e. $\eta(k)=0$. Let $F_0:=k$. As long as the Poisson clocks attached to the sites $F_t$ and $F_t-1$ do not ring, the front process does not jump. If there is a legal ring (see Definition~\ref{GraphRepr}) at $F_{t^-}$ at time $t$ \emph{and} the associated Bernoulli variable is a $1$ ({i.e.} $F_{t^-}$ is filled with a particle   { at time $t$}), set $F_t:=F_{t^-}+1$. Note that since the ring is assumed to be legal, the configuration at site $F_t$ is still a hole. If there is a (necessarily legal) ring at $F_{t^-}-1$ at time $t$ \emph{and} the associated Bernoulli variable is a $0$ ({i.e.}   {the site $F_{t^-}-1$ has a hole at time $t$}), set   {$F_t:=F_{t-}-1$}. Note that this front process is always on a zero of the configuration ($\eta_t(F_t)=0$ for all $t\geq 0$).\footnote{In \cite{O2} we actually start with a configuration entirely filled to the left of the initial position of the front. The front is then at any time the left-most zero of the system. Due to the orientation of the East dynamics, however, the above definition gives a process with exactly the same properties.} Lemma 3.2 in \cite{O2}, together with Borel-Cantelli lemma, says that the front moves at least with negative linear velocity asymptotically. More precisely, there exists a constant $\underline{v}>0$ such that
\begin{equation}\label{negfront}
\limsup_{t\rightarrow \infty} \frac{F_t}{t}\leq -\underline{v} \quad\text{a.s.}
\end{equation}

It is not difficult to see, using the same graphical construction of the underlying East dynamics for the degenerate walker and the front process, that if we choose initially
\begin{equation*}
F_0=Y_0+1/2,
\end{equation*}
then, for all time $t\geq 0$,
\begin{equation*}
F_t\geq Y_t+1/2.
\end{equation*}
In view of this coupling and \eqref{negfront}, the thesis is readily obtained.

\subsection{Proof of Proposition~\ref{denso}}\label{proofdenso}

 By Theorem \ref{chorale}, for any   { $f\in L^2(\mu )$}, we have that
\begin{equation}\label{sotto}
\mu_\e(f)=\mu(f)+\int_0^\infty \mu\left(\hat L_\e S(s)f\right)ds+O(\e^2).
\end{equation}
Now note that by definition $S(t)f(\eta)=\bbE_\eta^{(0)}\left[f(\eta_t)\right] =\sum_{y\in\bbZ}   {p_t(y)}\bbE^{\rm env}_\eta\left[f(\tau_y\eta_t)\right]$. 
By means of this observation, together with the definition of $\hat L_\e$ and translation invariance, we can write
\begin{eqnarray*}\notag
\mu \left(\hat L_\e S(s)f\right)&=&\e\mu \left(\left(2\eta(0)-1\right)\left[S(s)f(\tau_1\eta)-S(s)f(\tau_{-1}\eta)\right]\right)\\ \notag
&=&\e\mu \left(\left(2\eta(0)-1\right)\bbE_\eta^{(0)}\left[f(\tau_1\eta_s)-f(\tau_{-1}\eta_s)\right]\right)\\ 
&=&2\e\sum_{y\in\bbZ}   {p_s(y)}\mu\left(\eta(0)\bbE^{\rm env}_\eta\left[f(\tau_{y+1}\eta_s)-f(\tau_{y-1}\eta_s)\right]\right).
\end{eqnarray*}
By combining the expression above with \eqref{sotto}, we get \eqref{boh}.

\subsection{Proof of Corollary~\ref{oscar}}\label{proofoscar}
Equation \eqref{limitdensity} follows by plugging the function $f(\eta):=\eta(x)$ in \eqref{boh}   {as we explain}. The integrand in \eqref{boh} then equals
\begin{equation}\label{marcello}
\sum_{y\in\bbZ}p_s(y)\nu_{\rho}\left(\xi(0)\bbE^{\rm east}_\xi\left[\xi_s(x+y+1)-\xi_s(x+y-1)\right]\right).
\end{equation}
To simplify the above expression we first observe  that  $\nu_{\rho}\left(\xi(0)\bbE^{\rm east}_\xi\left[\xi_{s}(y)\right]\right)=\rho^2$ for any $y\neq 0$ (cf. Lemma \ref{orientation}   {below}). So that
\begin{eqnarray*}
\eqref{marcello}&=&\rho^2\sum_{y\neq -x-1}p_s(y)+p_s(-x-1)\nu_{\rho}\left(\xi(0)\bbE^{\rm east}_\xi\left[\xi_s(0)\right]\right)\\
&&\quad-\,\rho^2\sum_{y\neq -x+1}p_s(y)-p_s(-x+1)\nu_{\rho}\left(\xi(0)\bbE^{\rm east}_\xi\left[\xi_s(0)\right]\right)\\
&=&\left[p_s(x-1)-p_s(x+1)\right]\left[\rho^2-\nu_{\rho}\left(\xi(0)\bbE^{\rm east}_\xi\left[\xi_s(0)\right]\right)\right]\,.
\end{eqnarray*}
  {Recalling that  $u(s) =\rho^2-\nu_{\rho}\left(\xi(0)\bbE^{\rm east}_{\xi}\left[\xi_s(0)\right]\right)$ and coming back to \eqref{boh},} one gets \eqref{limitdensity}.

\medskip
  {It remains to analyze the function $u$}.
Clearly, $u(0)=-\rho(1-\rho)$. 
Moreover, setting $g(\eta):= \eta(0)-\rho$,    by reversibility  we get
\[ |u(s)|= |\nu_\rho \left( g (\xi) \bbE_\xi ^{\rm east} [g(\xi_s)] \right) |=
\nu_\rho(  [S^{\rm east}(s/2) g) ]^2) \leq e^{- s\l} \nu_\rho (g^2) = \rho (1-\rho) e^{-s \l}
\]
where $\l$  and  $S^{\rm east}(\cdot)$ denote respectively the  spectral gap  and the Markov semigroup of the East process. 

Finally, let us focus on the sign and the growth of $u$. Call $T_0$ the time of the first \emph{legal ring} at $0$, see Definition \ref{GraphRepr}.
\begin{equation}\label{aduc}
\begin{split}
\nu_{\rho}\left(\xi(0)\bbE^{\rm east}_\eta[\xi_s(0)]\right)&= \nu_{\rho}\left(\xi(0)\bbE^{\rm east}_\xi[\xi_s(0)|T_0\leq s]\mathbb{P}^{\rm east}_\xi(T_0\leq s)\right)\\
&+\,\nu_{\rho}\left(\xi(0)\bbE^{\rm east}_\xi[\xi_s(0)|T_0> s]\mathbb{P}^{\rm east}_\xi(T_0> s)\right)\,.
\end{split}
\end{equation}
On the one hand, if $T_0>s$, $\xi_s(0)=\xi_0(0)$ (if there has been no legal ring at the origin, by definition the configuration at $0$ has not been updated). So that the second term is $\nu_{\rho}\left(\xi(0)\mathbb{P}^{\rm east}_\xi(T_0> s)\right)$, which in turn is just $\rho\mathbb{P}^{\rm east}_{\nu_{\rho}}(T_0> s)$, since the event $\lbrace T_0>s\rbrace$ does not depend on $\xi(0)$. On the other hand, the first term can be rewritten as $\nu_{\rho}\left(\xi(0)\rho\mathbb{P}^{\rm east}_\xi(T_0\leq s)\right)$    { by distinguishing the zero that is on site $1$ at time $T_0$ (the fact that $T_0$ is the time of a legal ring at $0$ ensures its existence,
 cf. \cite{AD,CMRT}   for the definition of the distinguished zero and for its properties).} 
 Again by orientation of the model,   {the last mean}  is equal to   {$\rho^2\mathbb{P}^{\rm east}_{\nu_\rho} (T_0\leq  s)$}.   { Due to the previous observations, \eqref{aduc} equals
 $ \rho(1-\rho) \mathbb{P}^{\rm east}_\xi(T_0> s)+ \rho^2 $, hence 
 \begin{equation}
u(s)=-\rho(1-\rho)\mathbb{P}^{\rm east}_{\nu_{\rho}}(T_0> s)
\end{equation}
 } 
  It remains to notice that $\mathbb{P}^{\rm east}_{\nu_{\rho}}(T_0> s)$ is a quantity decreasing with $s$, so that
   {$u(s)$}
is indeed negative and increasing in $s$.

\appendix

\section{Useful facts on the East model}\label{A}
\begin{definition}{\bf (Graphical representation of the East model)}\label{GraphRepr}
Starting from a configuration $\eta\in\Omega$, the East dynamics $(\eta_t)_{t\geq 0}$ can be constructed as follows. With every $x\in\bbZ$ independently we associate a Poisson process with parameter $1$ that will be called the (Poisson) clock at $x$. The process can then be constructed in the following   {way}:
\begin{itemize}
\item Check the constraint: if the clock at site $x$ rings at time $t$, look at the constraint at $x$ in $\eta_t$, the configuration at time $t$.
\item If $c^{\text{east}}_x(\eta_t)=1$, the constraint is satisfied and the occupation variable at site $x$ is replaced by a Bernoulli variable of parameter $\rho$ independent of all the rest. The ring at time $t$ is said to be a \emph{legal ring}.
\item If $c^{\text{east}}_x(\eta_t)=0$, the constraint is not satisfied and the system is left unchanged.
\end{itemize}
\end{definition}

The following lemma is a consequence of reversibility 
and the orientation property of the East model which we use to prove Proposition~\ref{shaun}.
  { We write  $\bbE^{\rm east}_{\eta}$ for  the expectation of the East dynamics starting at $\eta$, and  we define $\bbE^{\rm east}_{\nu_\rho} $ similarly}.
\begin{Lemma}\label{orientation}
The following  holds:
\begin{enumerate}
\item Let $0\leq t_1\leq t_2\cdots \leq t_k$ and   {let  $f_1,...,f_k$ be  functions on $\{0,1\}^\bbZ$  
such that   the    convex envelopes  of their supports  ${\rm Conv}( Supp(f_{1}) ), \dots, {\rm Conv} ( Supp(f_k) )$  are disjoint. }
Then 
\begin{eqnarray*}
\bbE^{\rm east}_{\nu_\rho}\left[f_1\left(\eta_{t_1}\right)\ldots f_k\left(\eta_{t_k}\right)\right]=\prod_{i=1}^k\nu_\rho\left(f_i\right).
\end{eqnarray*} 
\item Let $0\leq t_1\leq t_2 \leq \cdots \leq t_k$, $f_1,...,f_k$    functions on    {$\{0,1\}^\bbZ$} and   $i_0\in\{1,...,k\}$ such that    { $\nu_\rho(f_{i_0})=0$ and $x <y$  for all $ x\in Supp(f_{i_0})$ and $   y\in\cup_{i\neq i_0}Supp(f_i)$} ({i.e.} the support of $f_{i_0}$ is to the left of all the other supports). Then
\begin{equation}
\bbE^{\rm east}_{\nu_\rho}\left[f_1\left(\eta_{t_1}\right)\ldots f_k\left(\eta_{t_k}\right)\right]=0.
\end{equation}
\end{enumerate}
\end{Lemma}
\begin{proof}
The first statement is a consequence of the second one by iteration (let $i_0$ be the index of the function with left-most support and apply the second statement replacing $f_{i_0}$ by $f_{i_0}-\nu_\rho(f_{i_0})$).

%
 Notice that by    reversibility we can construct the process at equilibrium also for negative times by mirroring the graphical construction. The process obtained is invariant by time translation. In particular, we have 
\begin{eqnarray*}
\bbE^{\rm east}_{\nu_\rho}\left[f_1\left(\eta_{t_1}\right)\ldots f_k\left(\eta_{t_k}\right)\right]&=&\bbE^{\rm east}_{\nu_\rho}\left[f_1\left(\eta_{t_1-t_{i_0}}\right)\ldots f_k\left(\eta_{t_k-t_{i_0}}\right)\right]\\
&=&\nu_\rho\Big (f_{i_0}(\eta)\bbE^{\rm east}_{\eta}\Big[\prod_{i\neq i_0}f_i\Big(\eta_{t_i-t_{i_0}}\Big)\Big]\Big)\,.
\end{eqnarray*}
Now notice that $\bbE^{\rm east}_{\eta}\left[\prod_{i\neq i_0}f_i\left(\eta_{t_i-t_{i_0}}\right)\right]$ has disjoint support from $f_{i_0}$ thanks to the orientation property of the East model. The two terms  in the    {$\nu_\rho$--mean} are therefore decorrelated. Hence the result.
\end{proof}

{\bf Acknowledgements}. The authors thank P. Thomann for providing the numerical experiments in Figures \ref{pSpeed} and \ref{dens0.5}.
L. Avena has been supported by NWO Gravitation Grant 024.002.003-NETWORKS.

\end{document}